\documentclass{amsart}
\usepackage[T1]{fontenc}
\usepackage{a4wide,mathtools,bbm,amsthm,amsmath,amsfonts}
\usepackage{amssymb,color,epic,multirow,comment}
\usepackage{algorithmicx,algorithm,pifont}
\usepackage{graphicx}
\usepackage{subcaption}
\usepackage[export]{adjustbox}
\usepackage{tikz,pgfplots}
\usepgfplotslibrary{groupplots}
\pgfplotsset{compat=newest}
\usetikzlibrary{arrows,decorations,backgrounds,positioning}
\usepgfplotslibrary{colorbrewer}
\theoremstyle{plain}

\newtheorem{theorem}{Theorem}[section]
\newtheorem{lemma}[theorem]{Lemma}

\newtheorem{corollary}[theorem]{Corollary}
\newtheorem{definition}[theorem]{Definition}
\newtheorem{remark}[theorem]{Remark}

\def\Letters{A,B,C,D,E,F,G,H,I,J,K,L,M,N,O,P,Q,R,S,T,U,V,W,X,Y,Z}
\makeatletter
\@for \@l:=\Letters \do{%
  \expandafter\edef\csname\@l bb\endcsname{%
  \noexpand\ensuremath{\noexpand\mathbb{\@l}}}%
  \expandafter\edef\csname\@l bf\endcsname{{\noexpand\bf \@l}}%
  \expandafter\edef\csname\@l cal\endcsname{%
  \noexpand\ensuremath{\noexpand\mathcal{\@l}}}%
  \expandafter\edef\csname\@l eu\endcsname{%
  \noexpand\ensuremath{\noexpand\EuScript{\@l}}}%
  \expandafter\edef\csname\@l frak\endcsname{%
  \noexpand\ensuremath{\noexpand\mathfrak{\@l}}}%
  \expandafter\edef\csname\@l rm\endcsname{{\noexpand\rm \@l}}%
  \expandafter\edef\csname\@l scr\endcsname{%
  \noexpand\ensuremath{\noexpand\mathscr{\@l}}}%
}
\newcommand{\bs}[1]{{\boldsymbol#1}}
\renewcommand{\d}{\operatorname{d\!}}

\newcommand{\isdef}{\mathrel{\mathrel{\mathop:}=}}

\newcommand{\dist}{\operatorname{dist}}

\newcommand{\spn}{{\operatorname{span}}}

\DeclareMathOperator*{\argmin}{\operatorname{argmin}}

\DeclareMathOperator{\kernel}{\Kcal}
\definecolor{navy}{RGB}{102,153,255}
\definecolor{tuerkis}{RGB}{51,153,204}
\algrenewcommand\alglinenumber[1]{\ding{\numexpr191 + #1}}

\title[The dimension weighted fast multipole method]
{The dimension weighted fast multipole method for scattered data approximation}
\author{Helmut Harbrecht, Michael Multerer, and Jacopo Quizi}
\email{}
\address{Helmut Harbrecht,
Departement f\"ur Mathematik und Informatik, 
Universit\"at Basel, 
Spiegelgasse 1, 4051 Basel, Schweiz.}
\email{helmut.harbrecht@unibas.ch}
\address{
Michael Multerer and Jacopo Quizi,
Istituto Eulero,
Universit{\`a} della Svizzera italiana,
Via la Santa 1, 6962 Lugano, Svizzera.}
\email{\{michael.multerer,jacopo.quizi\}@usi.ch}

\begin{document}
\begin{abstract}
The present article is concerned scattered data approximation
for higher di\-men\-sio\-nal data sets which exhibit an anisotropic
behavior in the different dimensions.
Tailoring sparse polynomial interpolation to this specific situation, we derive
very efficient degenerate kernel approximations which we then use in 
a dimension weighted fast multipole method. This dimension weighted
fast multipole method enables to deal with many more 
dimensions than the standard black-box fast multipole method based on
interpolation. A thorough analysis of the method is provided including rigorous
error estimates.
The accuracy and the cost of the approach are validated 
by extensive numerical results. As a relevant application, we apply
the approach to a shape uncertainty quantification problem.
\end{abstract}

\maketitle
\section{Introduction}\label{sec:introduction}
One of the main challenges in the computation of scattered data approximants
in reproducing kernel Hilbert spaces is the assembly and the solution of 
the associated linear system of equations for globally supported kernels.
In this case, the kernel matrix is densely populated resulting in 
a quadratic assembly- and a cubic solution cost with respect to the number of
data sites. This obstruction can be overcome by using suitable acceleration 
techniques, several of which have been developed in recent years.
These comprise approaches that are based on wavelet matrix compressions,
see, e.g., \cite{GBD98,Samplets}, approaches based on the fast multipole 
method in, see, e.g., \cite{GR87,MXTCB15,Tausch,YBZ04}, and approaches 
based on hierarchical matrices, see, e.g., 
\cite{AFGHO16,BG07,DHS17,FKS18,H15,LSGK19,PZ}.

All aforementioned methods have in common that they reduce the cost for the 
solution of the scattered data approximation problem to an essentially linear
one. However, each of these methods is subject to the 
\emph{curse of dimensionality} when the data sites become higher dimensional.
Nonetheless, there are many situations where the data sites are high-dimensional
but equipped with dimension weights, which means that particular dimensions are 
less important than others. Relevant examples arise from uncertainty 
quantification in physical models. Having sufficiently fast decaying 
dimension weights, numerical interpolation and quadrature
schemes can be designed which converge at rates that are not or only mildly 
subject to the curse of dimensionality, see, e.g., 
\cite{C2S15,CD15,DGLS17,HHPS18} and the references therein. 

In this article, we address the construction of a black-box fast multipole 
method tailored to higher dimensions in the presence of dimension weights. 
This approach differs from existing (data-driven) ones, see, e.g.\ 
\cite{CHCX,ASKIT,YLRB17}, in that we consider 
degenerate kernel approximant based on anisotropic polynomial interpolation 
in the kernel's farfield. This approach allows for rigorous a-priori
compression error estimates. Our estimates exploit the anisotropy induced
by the dimension weights under the assumption that the kernel's farfield 
is analytically extendable into a suitable tensor product domain. 
This \emph{weighted total degree polynomial interpolation} can be seen as 
an extension
of the sparse Smolyak interpolation on multidimensional hypercubes, 
see \cite{SMO63}. By working along the lines of what has been done in 
\cite{HHPS18} within the context of quadrature problems, we
derive error relative estimates that are independent of 
the number of data sites and decrease exponentially with 
respect to the chosen maximum polynomial degree for the interpolation. 
Having fixed a polynomial space for interpolation, we equip it with
an anisotropic Legendre basis and employ
approximate Fekete points, see \cite{BDME16, BDMSV11, BDMSV10}, 
subsampled from the Halton sequence, see \cite{Hal60}, 
for the interpolation of the farfield.

Our key contributions are the proof of the existence of 
an analytic extension of asymptotically smooth kernel functions into
anisotropic tensor product domains in case of dimension weights and the
proof of an error estimate for polynomial approximation is derived under the 
condition that the inverses of the radii of the analytic extension form an 
\(\ell^{1}\)-sequence. In particular, the generic constant of this error 
estimate is independent of the spatial dimension. 
Moreover, we propose a numerical realization of this approach based on
the interpolation with an anisotropic Legendre basis in combination with
approximate Fekete points.
Our results particularly show that the cost of the fast multipole
method can enormously be reduced, while the errors stays
controllable by the rigorous error estimates provided. In our numerical 
tests, we study the cost and the and accuracy of the dimension 
weighted fast multipole method for different weight sequences. 
Moreover, we demonstrate the feasibility of the present approach in 
practice by its application to a shape uncertainty quantification
problem.

The remainder of this article is arranged as follows: Section~\ref{sec:problem} 
recalls the concept of reproducing kernel Hilbert spaces and introduces 
the scattered data approximation problem within the corresponding context. 
Section~\ref{sec:anisoKernel} is dedicated to the properties 
of kernels on anisotropic data sets. In particular, we show
how such kernels can be extended analytically under the asymptotical 
smoothness assumption. In Section~\ref{sec:approximation}, we 
provide error estimates for the polynomial 
approximation of an analytic function. More specifically, 
through the introduction of polynomial projection operators and the 
corresponding error analysis, we provide an error estimate 
for a generic analytical function that can be analytically 
extended and we derive convergence of the dimension weighted
total degree interpolation. Section~\ref{sec:multipole}
is concerned with applying what has been proven in the 
preceding sections to develop the dimension weighted fast 
multipole method in the context of  
Section~\ref{sec:anisoKernel}. Section~\ref{sec:numerix} 
is dedicated to numerical results. Finally, in Section
\ref{sec:conclusio}, we state concluding remarks.

Throughout this article, to avoid the repeated use of
generic but unspecified constants, by \(C\lesssim D\)
we indicate that $C$ can be bounded by a multiple of $D$,
independently of parameters which $C$ and $D$ may depend on.
Moreover, \(C\gtrsim D\) is defined as \(D\lesssim C\)
and \(C\sim D\) as \(C\lesssim D\) and \(D\lesssim C\).

\section{Problem statement}\label{sec:problem}
We start by recalling the concept of \emph{reproducing kernel Hilbert spaces}. 
The interested reader finds a comprehensive representation of the underlying 
mathematical theory in \cite{Fasshauer2007,Wendland2004} for example.

\begin{definition}\label{def:RKHS}
Let \(\Omega\subset\Rbb^d\) and 
$\big(\mathcal{H},(\cdot,\cdot)_{\mathcal{H}}\big)$ be a Hilbert space of 
functions \(h\colon\Omega\to\Rbb\). A \emph{reproducing kernel} $\kernel$ for 
\(\Hcal\) is a function $\kernel:\Omega\times\Omega\to\Rbb$ such that
\begin{enumerate}
  \item $\kernel(\cdot,{\bs x})\in\mathcal{H}$ for all 
  ${\bs x}\in\Omega$,
  \item $h({\bs x}) = \big(h, \kernel (\cdot,{\bs x})\big)_\mathcal{H}$ 
  for all $h\in\mathcal{H}$ and all ${\bs x}\in\Omega$.
\end{enumerate}
If \(\Hcal\) exhibits a reproducing kernel, we call it a
\emph{reproducing kernel Hilbert space (RKHS)}. 
\end{definition}

The kernel of an RKHS is 
known to be symmetric and positive semidefinite
in the following sense.

\begin{definition}\label{def:poskernel}
A kernel
$\kernel\colon\Omega\times\Omega\rightarrow\Rbb$ is called 
\emph{positive (semi-)definite} on $\Omega\subset\Rbb^d$,
iff \([\kernel({\bs x}_i,{\bs x}_j)]_{i,j=1}^N\)
is a symmetric and positive (semi-)definite matrix
for all $\{{\bs x}_1, \ldots,{\bs x}_N\}\subset\Omega$
and all $N\in\mathbb{N}$.
\end{definition}

Given a set of \emph{data sites}
\(X=\{{\bs x}_1,\ldots,{\bs x}_N\}\subset\Rbb^d\)
and \emph{data values} \(y_1,\ldots,y_N\in\Rbb\),
the present article is concerned with the efficient computation of
the kernel interpolant 
\[
  s_X({\bs x})\isdef \sum_{j=1}^N \alpha_j\kernel({\bs x}_j,{\bs x})
\]
such that
\[
s_X({\bs x}_i)=y_i\quad\text{for }i=1,\ldots,N.
\]
This interpolant 
is given by solving the linear system of equations
\begin{equation}\label{eq:LGS0}
  {\bs K}{\bs \alpha} = {\bs y},
\end{equation}
where
\[
  {\bs K}\isdef[\kernel({\bs x}_i,{\bs x}_j)]_{i,j=1}^N\in\Rbb^{N\times N},
  \quad {\bs y}\isdef[y_i]_{i=1}^N\in\Rbb^{N},
  \quad {\bs \alpha}\isdef[\alpha_i]_{i=1}^N\in\Rbb^{N}.
\]
Closely related to this interpolation problem is the
\emph{kernel ridge regression}
\[
\min_{\alpha_1,\ldots,\alpha_N}\sum_{i=1^n}\big(y_i-s_X({\bs x}_i)\big)^2
+\lambda\|s_X\|_{\Hcal}^2,\quad\lambda>0,
\]
which regularizes the solution to the interpolation problem by
penalizing the norm of the interpolant. The first order condition
of this minimization problem yields the linear system
\begin{equation}\label{eq:LGS}
({\bs K}+\lambda{\bs I}){\bs \alpha} = {\bs y},
\end{equation}
which coincides with \eqref{eq:LGS0} for \(\lambda=0\).

As the \emph{kernel matrix} ${\bs K}$ is in general densely
populated, we shall apply the fast multipole method to
accelerate matrix-vector multiplications in an iterative solver
for the solution of the linear system~\eqref{eq:LGS}. 
The particular multipole method we consider 
exploits kernel expansions obtained by interpolation. That way, 
the approach is black box and applies also to non-stationary 
kernels, see \cite{Borm,Gie01} for example. In order to
deal with possibly high dimension, we use interpolation by 
(anisotropic) total degree polynomials.

\begin{figure}[htb]
\begin{center}
\begin{tikzpicture}[x=0.5cm,y=0.5cm]
  \shade[xslant=1,bottom color=gray!10, top color=gray!5]
    (-4,4) rectangle (4,5);
      \shade[top color=gray!10, bottom color=black!20]
    (0,0) rectangle (8,4);
      \shade[yslant=1,top color=gray!5, bottom color=black!20]
    (8,-8) rectangle (9.01,-4);
\draw[step=1](0,0) grid (8,4);
\draw[xslant=1,xstep=1](-4,4) grid (4,5);
\draw[xslant=1](-4,4.53)--(4,4.53);
\draw[xslant=1](-4,5)--(4,5);
\draw[yslant=1,ystep=1](8,-8) grid (9.01,-4);
\draw[yslant=1](9.01,-8)--(9.01,-4);
\draw[yslant=1](8.53,-8)--(8.53,-4);
\draw(4,-0.5)node{$b_1$};
\draw(-0.5,2)node{$b_2$};
\draw(-0.1,4.65)node{$b_3$};
\end{tikzpicture}
\caption{\label{fig:anisoBox}Considering quasi-uniform
data sites in an anisotropic bounding box yields isotropic clusters.
As a consequence, the nearfield becomes feasible.}
\end{center}
\end{figure}
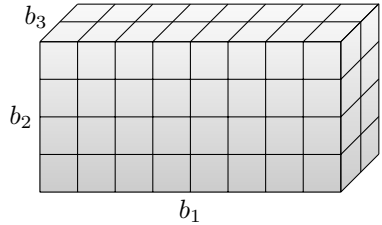

Throughout this article, we assume that the set \(X\) is \emph{quasi-uniform}
and contained in an anisotropic axis parallel cuboid 
\begin{equation}\label{eq:anisotropicBox}\Bcal=[0,b_1]
\times\cdots\times[0,b_d],
\end{equation} i.e., \(X\subset\Bcal\) with \emph{dimension weights} 
\(b_{1}\geq\ldots\geq b_d\geq 0\), see Figure~\ref{fig:anisoBox} for a 
visualization. 

\begin{definition}\label{def:qunifpoints}
The data set \(X=\{{\bs x}_1,\ldots,{\bs x}_N\}\subset\Omega\) 
is quasi-uniform if the \emph{fill distance} 
\[h_{X,\Omega}\isdef\sup\limits_{\bs{x}\in \Omega}
\min\limits_{\bs{x}_i\in X}\|\bs{x}-\bs{x}_i\|_2
\] 
is proportional to the separation radius 
\[
q_{X}\isdef\min\limits_{i\neq j}\|\bs{x}_i-\bs{x}_j\|_2,
\] 
i.e., there exists a constant \(c_{X,\Omega}\in (0,1)\) such that
\[
0<c_{X,\Omega}\leq\frac{q_X}{h_{X,\Omega}}\leq\frac{1}{c_{X,\Omega}}.
\]
\end{definition}

The condition \eqref{eq:anisotropicBox}is, for example,
satisfied if we are given quasi-uniform points in $X\subset\Rbb^d$ which 
exhibit an anisotropic structure in the sense of rapidly decaying 
singular values of the respective data matrix ${\bf X}
\isdef[{\bs x}_1,\ldots,{\bs x}_N]\in\Rbb^{d\times N}$. Then,
employing a principal component analysis and a suitable affine
transform, the data becomes uniformly distributed in the set's
bounding box given by the anisotropic 
cuboid \(\Bcal\). Another important example are samples of partial
differential equations with random inputs, where the latter is
represented by a Karhunen-Lo{\`e}ve expansion. In this case, the
weights \(b_i\) correspond to the singular values up to the 
\(L^\infty\)-normalization of the spatial eigenfunctions.

In what follows, we study how the anisotropy, imposed by the 
dimension weights $\{b_i\}_{i=1}^d$, affects the 
approximation properties of the
fast multipole method. In the first place, we observe
that a cardinality based clustering results in isotropic 
clusters as the points are uniformly distributed in $\mathcal{B}$.
Hence, as an immediate consequence, the nearfield becomes feasible,
compare Figure~\ref{fig:anisoBox}. But also for the kernel
approximation in the farfield we can use different polynomial 
degrees for each dimension to account for the different 
dimension weights. The respective analysis is carried 
out in the following sections.

\section{Kernels on anisotropic data sets}\label{sec:anisoKernel}
For our analysis, we introduce the 
linear transformation 
\[{\bs B}\colon[0,1]^d\to\Bcal,\quad
\hat{\bs x}\mapsto{\bs x}={\bs B}\hat{\bs x},
\]
which is given by the matrix
\[
{\bs B}\isdef\operatorname{diag}(b_1,\ldots,b_d).
\]
Thus, considering the kernel as a bivariate function
\[
\kernel\colon\Bcal\times\Bcal\to\Rbb
\]
gives rise to the \emph{transported kernel} 
\begin{equation}\label{eq:transKernel}
\kernel_{\bs B}\colon[0,1]^d\times[0,1]^d\to\Rbb,\quad
\kernel_{\bs B}(\hat{\bs x},\hat{\bs y})
\isdef\kernel({\bs B}\hat{\bs x},{\bs B}\hat{\bs y}).
\end{equation}

The subsequent analysis of 
the kernel interpolation is based on
the \emph{asymptotical smoothness}
property of the kernel \(\kernel\), that is
\begin{equation}\label{eq:kernel_estimate}
  \bigg|\frac{\partial^{|\bs\alpha|+|\bs\beta|}}
  	{\partial{\bs x}^{\bs\alpha}
  	\partial{\bs y}^{\bs\beta}} \kernel({\bs x},{\bs y})\bigg|
  		\le c_{\kernel} \frac{(|\bs\alpha|+|\bs\beta|)!}
		{\rho^{|\bs\alpha|+|\bs\beta|}
		\|{\bs x}-{\bs y}\|_2^{|\bs\alpha|+|\bs\beta|}},\quad 
		c_{\kernel},\rho>0.
\end{equation}

For the derivatives of the transported kernel function~\eqref{eq:transKernel}, 
we obtain the following anisotropic 
version of the asymptotic smoothness property \eqref{eq:kernel_estimate}.

\begin{lemma}\label{lem:transKernelDer} 
The derivatives of the transported kernel satisfy the bound
\begin{equation}\label{eq:trans_kernel_estimate}
  \bigg|\frac{\partial^{|\bs\alpha|+|\bs\beta|}}
  	{\partial\hat{\bs x}^{\bs\alpha}
  	\partial\hat{\bs y}^{\bs\beta}}\kernel_{\bs B}
	(\hat{\bs x},\hat{\bs y})\bigg|
  		\le c_{\kernel}\bigg(
		\frac{\bs b}{\rho}\bigg)^{{\bs\alpha}+{\bs\beta}}
		\frac{(|\bs\alpha|+|\bs\beta|)!}
		{
		\|{\bs B}(\hat{\bs x}-\hat{\bs y})\|_2^{|\bs\alpha|+|\bs\beta|}},
		\quad 
		c_{\kernel},\rho>0
\end{equation}
with \({\bs b}\isdef[b_1\ldots,b_d]\).
\end{lemma}

\begin{proof} 
By the chain rule, there holds
\[
  \frac{\partial^{|\bs\alpha|+|\bs\beta|}}
  	{\partial\hat{\bs x}^{\bs\alpha}
  	\partial\hat{\bs y}^{\bs\beta}} 
	\kernel_{\bs B}(\hat{\bs x},\hat{\bs y})
	=
	\frac{\partial^{|\bs\alpha|+|\bs\beta|}}
  	{\partial{\bs x}^{\bs\alpha}
  	\partial{\bs y}^{\bs\beta}} \kernel({\bs B}\hat{\bs x},
	{\bs B}\hat{\bs y}){\bs b}^{{\bs\alpha}+{\bs\beta}},
	\]
	since the mapping \(\hat{\bs x}\mapsto{\bs B}\hat{\bs x}\) is linear.
	From this, the claim is obtained by inserting the
	asymptotical smoothness property \eqref{eq:kernel_estimate}.
\end{proof}

The bound \eqref{eq:trans_kernel_estimate} allows to 
analytically extend each component of the transported 
kernel into an anisotropic region in the complex space 
\(\Cbb^d\). Following the structure of \cite[Lemma 4.76]{Borm}, 
we have the subsequent result.

\begin{lemma}\label{holomorphic_ext_nD}
Let \(f \in C([0,1]^d)\) and \(v_f\in [0,\infty)\)
be such that
\[
|\partial^{\bs\alpha}f(\boldsymbol{x})|
\leq {c_f}\frac{|\boldsymbol{\alpha}|!}{{\bs\rho}^{\bs\alpha}}\quad
\text{holds for all }{\bs\alpha} \in \Nbb^d,\ \boldsymbol{x}\in [0,1]^d.
\]
For \(\tau_k \in [0,\rho_k c_k]\), where \(c_k>0\) 
and \(\gamma\isdef\sum_{k=1}^dc_k<1\), there exists an analytic
extension
\(\tilde{f} \in C^\infty\big(\bs\Sigma({\bs\tau});\Rbb\big)\)
of \(f\) with
\[
\big|\tilde{f}(\bs {z})\big| \leq \frac{c_f}{1-\gamma}
\quad\text{for }{\bs z}\in\bs\Sigma({\bs\tau}),
\]
where we define
\(\bs\Sigma({\bs\tau})\isdef\Sigma(\tau_1)\times\cdots\times
\Sigma(\tau_d)\) with
\(\Sigma(\tau)\isdef\{z \in \Cbb :
\dist(z,[0,1]) \leq \tau \} \).
\end{lemma}

\begin{proof} 
For \(\bs{z_0} \in [0,1]^d\) arbitrary but fixed, we define 
\[
{D}_{\bs{z}_0} \isdef\{\bs{z} \in \mathbb{C}^d : |z_k-z_{0,k}|
\leq \tau_k\text{ for all }k\leq d\}.
\]
Thus, we have 
\(|z_k-z_{0,k}| \leq \tau_k \leq \rho_k c_k\) for all 
\(\bs{z} \in {D}_{\bs{z}_0}\). From this, we deduce that
\begin{align*}
\sum_{\boldsymbol{\alpha}\geq\boldsymbol{0}}\bigg|
\partial^{\boldsymbol{\alpha}}f(\bs{z_0})\frac{(\bs{z}-
\bs{z_0})^{\boldsymbol{\alpha}}}{\boldsymbol{\alpha}!}\bigg|
&\leq 
c_f\sum_{i=0}^{\infty} \sum_{|\boldsymbol{\alpha}|=i}
\bigg|\bigg(\frac{\bs{z}-\bs{z_0}}{\bs{\rho}}\bigg)^{\boldsymbol{\alpha}}
\bigg|\frac{|\boldsymbol{\alpha}|!}{\boldsymbol{\alpha}!}\leq
c_f\sum_{i=0}^{\infty} \sum_{|\boldsymbol{\alpha}|=i}
\bs{c}^{\bs{\alpha}}\frac{|\boldsymbol{\alpha}|!}{\boldsymbol{\alpha}!}\\
&= c_f\sum_{i=0}^{\infty} \bigg(\sum_{k=1}^{n}c_k\bigg)^i
\leq c_f\sum_{i=0}^{\infty}\gamma^i=\frac{c_f}{1-\gamma},
\end{align*}
where we used the multinomial theorem for the first equality.
This shows that the Taylor series at \(\bs{z}_0\) converges on 
\({D}_{\bs{z}_0}\) and therefore 
\[
\tilde{f}_{\bs{z_0}}\colon D_{\bs{z_0}} \rightarrow \Rbb, 
\quad \bs{z} \mapsto \sum_{\bs{\alpha}\geq \bs{0}}
\bigg|{\partial^{\bs{\alpha}}f(\bs{z}_0)
\frac{(\bs{z}-\bs{z}_0)^{\bs{\alpha}}}{\bs{\alpha}!}}\bigg| 
\]
is analytic. 

We next show that \(\tilde{f}_{\bs{z_0}}\) is an 
analytic extension of \(f\) on \({D}_{\bs{z_0}}\). 
There holds
\begin{align*}
&\bigg|f(\bs{z})-\sum_{|\bs{\alpha}|\leq k-1}{
\partial^{\bs{\alpha}}f(\bs{z})
\frac{(\bs{z}-\bs{z_0})^{\bs{\alpha}}}{\bs{\alpha}!}}\bigg|\\
&\qquad\leq k \sum_{|\bs{\alpha}|=k}\frac{|\bs{z}-\bs{z_0}|^{\bs{\alpha}}}
{\bs{\alpha}!}\left|\int_{0}^{1} (1-t)^{k-1}\partial^{\bs{\alpha}}
f\big(\bs{z_0}+t(\bs{z}-\bs{z_0})\big)\d t \right|\\
&\qquad\leq  c_f k \sum_{|\bs{\alpha}|=k}\frac{|\bs{z}-\bs{z_0}|^{\bs{\alpha}}}
{\bs{\alpha}!}\frac{|\boldsymbol{\alpha}|!}{{\bs\rho}^{\bs\alpha}}
\int_{0}^{1} (1-t)^{k-1}\d t  \leq
c_f\sum_{|\bs{\alpha}|=k}\bs{c}^{\bs{\alpha}}\frac{|\boldsymbol{\alpha}|!}
{{\bs\alpha}!} \\
&\qquad\leq c_f\gamma^k
\end{align*}
for all $k \in \Nbb$. Since \(\gamma < 1 \), if $k$ tends to infinity, 
we obtain \(\tilde{f}_{\bs{z_0}}(\bs{z})=f(\bs{z})\) for all \(\bs{z}\in 
D_{\bs{z_0}}\cap[0,1]^d\) and therefore \(\tilde{f}_{\bs{z_0}}\) is an 
analytic extension of \(f\) on \(D_{\bs{z_0}}\). Since \([0,1]^d\) 
is compact, there exist finitely many points such that \([0,1]^d
\subset\bigcup_{k=1}^m D_{{\bs z}_k}\). Patching together the 
extensions \(\tilde{f}_{{\bs z}_k}({\bs z})\) yields the desired 
analytic extension $\tilde{f}({\bs z})$.
\end{proof}

Inserting a lower bound \(\eta>0\) for the distance of now yields
the analyticity of the transported kernel for any
pair of points \(\hat{\bs x},\hat{\bs y}\in[0,1]^d\) 
satisfying \(\|{\bs B}(\hat{\bs x}-\hat{\bs y})\|_2\geq\eta\).

\begin{corollary}\label{cor:anisoKernelAdm}
Let \(\hat{\bs x},\hat{\bs y}\in[0,1]^d\) be such that 
\(\|{\bs B}(\hat{\bs x}-\hat{\bs y})\|_2\geq\eta\) for some 
\(\eta>0\). Then, each component of the kernel function 
\(\kernel_{\bs B}\) from \eqref{eq:transKernel} admits 
an analytic extension into \(\bs\Sigma({\bs\tau})\) for 
\(\tau_k\in[0,c_k\eta\rho/b_k]\), where \(c_k>0\) and 
\(\gamma\isdef\sum_{k=1}^dc_k<1\).
\end{corollary}

\begin{proof}
Due to Lemma~\ref{lem:transKernelDer}, there holds
\[
  \bigg|\frac{\partial^{|\bs\alpha|+|\bs\beta|}}
  	{\partial\hat{\bs x}^{\bs\alpha}
  	\partial\hat{\bs y}^{\bs\beta}}\kernel_{\bs B}
	(\hat{\bs x},\hat{\bs y})\bigg|
  		\le c_{\kernel}\bigg(\frac{\bs b}{\rho}\bigg)^{{\bs\alpha}+{\bs\beta}}
		\frac{(|\bs\alpha|+|\bs\beta|)!}
		{\|{\bs B}(\hat{\bs x}-\hat{\bs y})\|_2^{|\bs\alpha|+|\bs\beta|}}
		\leq c_{\kernel}\bigg(
		\frac{\bs b}{\eta\rho}\bigg)^{{\bs\alpha}+{\bs\beta}}
		(|\bs\alpha|+|\bs\beta|)!.
\]
From this, the assertion is directly obtained by setting 
either \({\bs\alpha}={\bs 0}\) or \({\bs\beta}={\bs 0}\) 
and applying Lemma~\ref{holomorphic_ext_nD}.
\end{proof}

\section{Weighted total degree polynomial approximation}
\label{sec:approximation}
To exploit the available anisotropy for approximation,
we derive corresponding estimates of the error for 
polynomial best approximation and extend them to the
interpolation case. To do so, we start from 
the following lemma that provides an error estimate for
the polynomial approximation of an analytic function. 
It follows directly from \cite[Chapter 7, \S 8]{DL93}.

\begin{lemma}\label{lem:analytic_bound} 
Let \(\Pi_{q}\isdef\spn\{1,\ldots,x^q\}\) denote the 
space of all univariate polynomials up to degree $q$.
Given a function \(f\in C([0,1])\) which admits an 
analytic extension \(\tilde{f}\) into the region 
\(\Sigma(\tau)\) for some \(\tau>0\), there holds,
for \(1<\rho\isdef2\tau+\sqrt{1+4\tau^{2}}\), that
\[
    \min_{w \in \Pi_{q}}\| f-w\|_{C([0,1])}\leq 
    \frac{2\rho}{\rho - 1}e^{-(q+1)\log\rho}\|\tilde{f}\|_{C(\Sigma(\tau))}.
\]
\end{lemma}

Especially, it is shown in \cite{DL93} that the approximating polynomial 
in Lemma~\ref{lem:analytic_bound} is unique. We may therefore define 
the projection operator
\begin{equation}\label{eq:projection}
U_q\colon C([0,1])\to\Pi_q,\quad U_qf
\isdef\argmin_{w\in\Pi_q}\|f-w\|_{C([0,1])}.
\end{equation}
By observing
\[
  (U_0f)(x) = \frac{1}{2}\bigg(\min_{y\in[0,1]} f(y)+\max_{y\in[0,1]}f(y)\bigg),
\]
we infer the stability estimate
\[
  \|U_0 f\|_{C([0,1])}\le\|f\|_{C([0,1])},
\]
i.e., the projection $U_0$ is stable with 
constant equal to 1.

With the projections onto the spaces \(\Pi_q\) at hand, we define 
the \emph{detail projections}
\[
\Delta_q\isdef U_{q}-U_{q-1},\quad \Delta_0\isdef U_0,\quad q=1,2,\ldots,
\]
which we will use in the subsequent construction of the 
weighted total degree approximation and the respective 
error. The next lemma gives a bounds the projection error of 
the detail projections.

\begin{lemma}\label{lem:detailInterpolation}
Given a function \(f\in C([0,1])\), which admits an 
analytic extension $\tilde{f}$ into the region 
\(\Sigma(\tau)\) for some \(\tau>0\), there holds for any \(q\in\Nbb\) that
\[
\|\Delta_qf\|_{C([0,1])}\leq 2^{\min\{1,q\}}c_{\rho}e^{-q\log\rho}
\|\tilde{f}\|_{C(\Sigma(\tau))}
\]
for \(1<\rho\isdef2\tau+\sqrt{1+4\tau^{2}}\) and
\[
c_\rho\isdef
\frac{\rho+1}{\rho-1}.
\]
\end{lemma}

\begin{proof}
For $q=0$, the desired estimate is obviously true by
\[
  \|\Delta_0 f\|_{C([0,1])}\leq\|f\|_{C([0,1])}
  \le\frac{\rho+1}{\rho-1}\|\tilde{f}\|_{C(\Sigma(\tau))}
\]
and $e^{-\log\rho} = 1/\rho$.
For $q\ge 1$, we conclude by using
Lemma~\ref{lem:analytic_bound} again that
\begin{align*}
\|\Delta_qf\|_{C([0,1])}&\leq\|U_{q}f-f\|_{C([0,1])}
+\|f-U_{q-1}f\|_{C([0,1])}\\
&\leq\frac{2\rho}{\rho-1}\big(e^{-\log\rho}+1\big)
e^{-q\log\rho}\|\tilde{f}\|_{C(\Sigma(\tau))}.
\end{align*}
Inserting the identity $e^{-\log\rho} = 1/\rho$ yields
the desired claim.
\end{proof}

The constant \(c_{\rho}\) introduced in Lemma~\ref{lem:detailInterpolation}
obviously satisfies \(c_{\rho}\searrow 1\) for \(\rho\to\infty\)
The next lemma  shows a bound on the product $\prod_{k=1}^dc_{\rho_k}$ of this 
constant for $d\to\infty$ when the sequence \(\{\rho_k^{-1}\}_k\) is summable. 
This bound will be used later in the multivariate case.

\begin{lemma}\label{lem:constbound}
Let \(\{\rho_k^{-1}\}_k\in\ell^1(\Nbb)\). Then there exists a constant
\(c_{\Pi}>0\), independent of \(d\), such that
\[
\prod_{k=1}^dc_{\rho_k}\leq c_\Pi.
\]
\end{lemma}

\begin{proof}
There holds
\[
\prod_{k=1}^dc_{\rho_k}=\prod_{k=1}^d\frac{\rho_k+1}{\rho_k - 1}
\leq\prod_{k=1}^\infty\frac{\rho_k+1}{\rho_k - 1}
=e^{\sum_{k=1}^\infty\log\frac{\rho_k+1}{\rho_k - 1}}.
\]
We infer that the product is bounded when 
\[
\sum_{k=1}^\infty\log\frac{\rho_k+1}{\rho_k - 1} 
= \sum_{k=1}^\infty\log\bigg(1+\frac{2}{\rho_k - 1}\bigg) <\infty.
\]
The sum converges by the limit comparison test provided that
\[
\sum_{k=1}^\infty\frac{2}{\rho_k - 1}
= 2\sum_{k=1}^\infty\frac{1}{\rho_k - 1}
\]
converges. In view of \(\{\rho_k^{-1}\}\in\ell^1(\Nbb)\), there 
exists an index \(k_0\) such that 
\[
\frac{1}{\rho_k-1}\leq\frac{2}{\rho_k}
\] 
or \(\rho_k\geq 2\) for all \(k>k_0\). Thus, we arrive at
\[
\sum_{k=1}^\infty\frac{1}{\rho_k - 1}=
\sum_{k=1}^{k_0}\frac{1}{\rho_k - 1}+
\sum_{k=k_0+1}^{\infty}\frac{1}{\rho_k - 1}
\leq \sum_{k=1}^{k_0}\frac{1}{\rho_k - 1}
+\sum_{k=k_0+1}^{\infty}\frac{2}{\rho_k}<\infty.
\]
Here, the first sum is bounded as it only contains 
finitely many terms and the second one is bounded due to 
\(\{\rho_k^{-1}\}\in\ell^1(\Nbb)\). This completes the proof.
\end{proof}

To define the weighted total degree approximation, 
we introduce the \emph{weighted total degree index sets}
\[
\Lambda_{\boldsymbol{\omega},q,d}\isdef
\bigg\{{\bf 0} \leq \boldsymbol{\alpha} \in \Nbb^{d}: 
\sum_{\ell=1}^d\omega_\ell\alpha_\ell \leq q\bigg\}
\]

for a weight vector \(\boldsymbol{\omega}\in [0,\infty)^d\).
The cardinality of the sets \(\Lambda_{\boldsymbol{\omega},q,d}\) has 
been estimated in \cite{HHPS18}. We particularly recall here
\cite[Lemma 5.5]{HHPS18}.

\begin{lemma}\label{lem:indexBound}
Let the weight vector \(\boldsymbol{\omega}
=[\omega_1,...,\omega_d]\) 
satisfy \(\omega_k\geq r\log k\) for some
constant \(r>1\) and let it be sorted in ascending order.
Then, the cardinality of the set \(\Lambda_{\boldsymbol{\omega},q,d}\)
is, for \(d>2\), bounded by 
\begin{equation}\label{eq:IndexSetBound}
|\Lambda_{\boldsymbol{\omega},q,d}|\leq c_r\log(d)^{\frac q r}
\end{equation}
with a constant \(c_r\) independent of \(d\).
\end{lemma}

Given the index set \(\Lambda_{\boldsymbol{\omega},q,d}\), we define
the associated polynomial space according to
\[
\Pcal_{\bs\omega,q}\isdef\operatorname{span}
\{{\bs x}^{\bs\alpha}:{\bs\alpha}\in \Lambda_{\boldsymbol{\omega},q,d}\}.
\]
Note 
that $\Pcal_{\bs\omega,q}$ coincides with the usual space 
of total degree polynomials up to degree $q$ if \(\omega_1
=\ldots=\omega_d=1\).
With respect to \(\Pcal_{\bs\omega,q}\), we introduce
the \emph{weighted total degree approximation operator 
of degree} \(q \in \Nbb\) given by
\[
P_{\boldsymbol{\omega},q}\colon C([0,1]^d)
\to\Pcal_{\bs\omega,q}, \quad
P_{\boldsymbol{\omega},q}f\isdef \sum_{\boldsymbol{\alpha} 
\in \Lambda_{\boldsymbol{\omega},q,d}}\Delta_{\bs\alpha}f.
\]

The estimate for the approximation error entailed by \(P_{\bs\omega,q}\)
is based on the polynomial analogue of dominating mixed smoothness
of the underlying function.
This mixed smoothness is exploited by the subsequent error estimate for the
tensor product of detail approximations.
For the sake of a lighter notation, in what follows, given an operator \(O\), we
define its \(k\)-fold tensor product according to
\[
O^{(k)}\isdef\underbrace{O\otimes\cdots\otimes O}_{\text{$k$-times}}.
\]

\begin{lemma}\label{lem:tpDetailError}
Given a function \(v\in C([0,1]^d)\) which admits an analytic
extension \(\tilde{v}\) into the region \(\bs\Sigma({\bs\tau})\) 
for \(0<\tau_1\leq\ldots\leq\tau_d\), there holds for 
\(1<\rho_k\isdef2\tau_k+\sqrt{1+4\tau_k^{2}}\), $k=1,\ldots,d$, that
\[
\|( \Delta_{\alpha_1}\otimes \cdots \otimes  \Delta_{\alpha_d})v\|_{C([0,1]^d)}
\leq \bigg(\prod_{k=1}^d
2^{\min\{1,\alpha_k\}}c_{\rho_k}\bigg)
e^{-\sum_{k=1}^d\alpha_k\log\rho_k}\|\tilde{v}\|_{C(\bs\Sigma({\bs\tau}))}.
\]
\end{lemma}

\begin{proof}
There holds by Lemma~\ref{lem:detailInterpolation}
\begin{align*}
&\|( \Delta_{\alpha_1}\otimes \cdots \otimes\Delta_{\alpha_d})v\|_{C([0,1]^d)}\\
&\qquad\leq 2^{\min\{1,\alpha_1\}}c_{\rho_1}e^{-\alpha_1\log\rho_1}\|(\operatorname{Id}\otimes\Delta_{\alpha_2} 
\cdots \otimes  \Delta_{\alpha_d})\tilde{v}\|_{C(\Sigma(\tau_1))}\\
&\qquad\leq 2^{\min\{1,\alpha_1\}}c_{\rho_1}2^{\min\{1,\alpha_2\}}c_{\rho_2}e^{-\alpha_1\log\rho_1-\alpha_2\log\rho_2}\\
&\hspace{4cm}\times
\|(\operatorname{Id}^{(2)}\otimes \Delta_{\alpha_3}\otimes\cdots\otimes  
\Delta_{\alpha_d})\tilde{v}\|_{C(\Sigma(\tau_1)\times\Sigma(\tau_2))}\\[-1ex]
&\qquad{}\ \,\vdots\\[-2ex]
&\qquad\leq \bigg(\prod_{k=1}^d
2^{\min\{1,\alpha_k\}}c_{\rho_k}\bigg)e^{-\sum_{k=1}^d\alpha_k\log\rho_k}
\|\tilde{v}\|_{C(\bs\Sigma({\bs\tau}))}.
\end{align*}
\end{proof}

The following theorem is the polynomial
best approximation version of a similar result
for quadrature in \cite{HHPS18}.
This theorem provides a bound
of the error of the polynomial best 
approximation. The generic constant \(c_\Pi\)
in this bound is independent of the
dimension $d$ under the condition
that the inverses of the convergence
radii form a summable sequence.

\begin{theorem}\label{th:errorEstimate}
Let \(v\in C([0,1]^d)\) admit an analytic
extension \(\tilde{v}\) into the region 
\(\bs\Sigma({\bs\tau})\) for a monotonously increasing sequence
\(0<\tau_1\leq\ldots\leq\tau_d\le\dots\)
and \(\{\tau_k^{-1}\}\in\ell^1(\Nbb)\).
Define
\(1<\rho_k\isdef2\tau_k+\sqrt{1+4\tau_k^{2}}\) for $k=1,\ldots,d$
and introduce the weights \(\omega_k\isdef\log\rho_k\).
Then, there holds
\[
\big\|(\operatorname{Id}^{(d)} - 
P_{\boldsymbol{\omega},q})v\big\|_{C([0,1]^d)}\leq
c_{\Pi}|\Lambda_{\boldsymbol{\omega},q,d-1}|2^{d+1}e^{-q}
\|\tilde{v}\|_{C(\bs\Sigma({\bs\tau}))}
\]
with the constant \(c_{\Pi}>0\) introduced in 
Lemma~\ref{lem:constbound}, being independent of \(d\).
\end{theorem}

\begin{proof}
We perform the proof along the lines of \cite{HHPS18,NTW08}.
The difference between the identity and the weighted total degree approximation
can be represented as
\begin{equation}\label{eq:errorRep}
\operatorname{Id}^{(d)} - P_{\boldsymbol{\omega},q} 
= \sum_{k=1}^d R(q,k)\otimes\operatorname{Id}^{(n-k)}.
\end{equation}
Herein, the quantity \(R(q,d)\), for \(n\geq 1\), is defined by
\[
R(q,1)\isdef\operatorname{Id}-U_{\iota_{{\bs\omega},q}({\bs 0},1)}
\]
for \(k=1\) and as
\[
R(q,k)\isdef
\sum_{\bs \alpha \in \Lambda_{\boldsymbol{\omega},q,k-1}}
\Delta_{\alpha_1}\otimes\cdots\otimes\Delta_{\alpha_{k-1}}
\otimes(\operatorname{Id}-U_{\iota_{{\bs\omega},q}({\bs\alpha},k)})
\]
for \(k\geq 1\), where we made use of the definition
\[
\iota_{{\bs\omega},q}({\bs\alpha},k)\isdef
\bigg\lfloor\frac{q-\sum_{\ell=1}^{k-1}\alpha_\ell \omega_\ell}{\omega_k}\bigg\rfloor.
\]

To prove the assertion, we estimate each summand of the error representation \eqref{eq:errorRep}
individually.
Due to \(2\rho_1/(\rho_1-1)\leq 2c_{\rho_1}\), we obtain
for \(k=1\) that
\[
\big\|R(q,1)\otimes\operatorname{Id}^{(n-1)}v\big\|_{C([0,1]^d)}\leq
 2 c_{\rho_1}e^{-(\iota_{{\bs\omega},q}({\bs 0},1)+1)\log\rho_1}
 \|\tilde{v}\|_{C(\bs\Sigma(\bs\tau))}
 \leq 2 c_{\rho_1} e^{-q}\|\tilde{v}\|_{C(\bs\Sigma(\bs\tau))}
\]
by Lemma~\ref{lem:analytic_bound}.

For \(k>1\), we employ Lemma~\ref{lem:analytic_bound},
Lemma~\ref{lem:constbound}, and
Lemma~\ref{lem:tpDetailError} to obtain
\begin{align*}
&\big\|R(q,k)\otimes\operatorname{Id}^{(n-k)}v\big\|_{C([0,1]^d)}\\
&\qquad\leq
\sum_{\bs\alpha \in \Lambda_{\boldsymbol{\omega},q,k-1}}
\big\|\Delta_{\alpha_1}\otimes\cdots\otimes\Delta_{\alpha_{k-1}}
\otimes(\operatorname{Id}-U_{\iota_{{\bs\omega},q}({\bs\alpha},k)})
\otimes\operatorname{Id}^{(n-k)}v\big\|_{C([0,1]^d)}\\
&\qquad\leq 2c_{\rho_k}\sum_{\bs\alpha \in \Lambda_{\boldsymbol{\omega},q,k-1}}
\bigg(\prod_{\ell=1}^{k-1}2^{\min\{1,\alpha_\ell\}}c_{\rho_\ell}\bigg)
e^{-\sum_{\ell=1}^{k-1}\alpha_\ell\log\rho_\ell}
e^{-(\iota_{{\bs\omega},q}({\bs\alpha},k)+1)\log\rho_k}
\|\tilde{v}\|_{C(\bs\Sigma({\bs\tau}))}\\
&\qquad\leq 2c_\Pi\bigg(\sum_{\bs\alpha \in \Lambda_{\boldsymbol{\omega},q,k-1}}
\prod_{\ell=1}^{k-1}2^{\min\{1,\alpha_\ell\}}
\bigg)e^{-q}\|\tilde{v}\|_{C(\bs\Sigma({\bs\tau}))},
\end{align*}
where we used \(\omega_\ell=\log\rho_\ell\) for \(\ell=1\ldots,k\)
in the second last step. Estimating the product by
the crude bound \(2^{k-1}\), we arrive at
\[
\big\|R(q,k)\otimes\operatorname{Id}^{(n-k)}v\big\|_{C([0,1]^d)}\leq
c_{\Pi} |\Lambda_{\boldsymbol{\omega},q,k-1}|2^ke^{-q}
\|\tilde{v}\|_{C(\bs\Sigma({\bs\tau}))}.
\]
The error estimate is now obtained by summation
via the geometric series according to
\begin{align*}
\big\|(\operatorname{Id}^{(d)} - P_{\boldsymbol{\omega},q})v\big\|_{C([0,1]^d)}
&\leq c_{\Pi}\sum_{k=1}^d|\Lambda_{\boldsymbol{\omega},q,k-1}|2^ke^{-q}
\|\tilde{v}\|_{C(\bs\Sigma({\bs\tau}))}\\
&\leq c_{\Pi} |\Lambda_{\boldsymbol{\omega},q,d-1}|2^{d+1}e^{-q}
\|\tilde{v}\|_{C(\bs\Sigma({\bs\tau}))}.
\end{align*}
\end{proof}

In view of Lemma~\ref{lem:indexBound}, we can estimate the number of degrees of freedom in the polynomial space $\Pcal_{\boldsymbol{\omega},q}$ by
\[
|\Lambda_{\boldsymbol{\omega},q,d-1}|\le c_r
 e^{\frac{q}{r}\log\log(d)}
\]
and Theorem~\ref{th:errorEstimate} implies the 
following estimate on the approximation error by total degree polynomials in 
case of algebraically increasing regions of analyticity.

\begin{corollary}\label{cor:withConstants}
Let \(v\in C([0,1]^d)\) admit an analytic
extension \(\tilde{v}\) into the region 
\(\bs\Sigma({\bs\tau})\) for a monotonously increasing sequence
\(0<\tau_1\leq\ldots\leq\tau_d\le\dots\)
and \(\tau_k\geq ck^r\) for some \(c,r>1\).
Define
\(1<\rho_k\isdef2\tau_k+\sqrt{1+4\tau_k^{2}}\) for $k=1,\ldots,d$
and introduce the weights \(\omega_k\isdef\log\rho_k\).
Then, there holds
\[
\big\|(\operatorname{Id}^{(d)} - P_{\boldsymbol{\omega},q})v\big\|_{C([0,1]^d)}
\leq
c_r2^{d+1} e^{-q(1-\frac 1 r\log\log d)}\|\tilde{v}\|_{C(\bs\Sigma({\bs\tau}))}
\]
with the constant \(c_r\) from \eqref{eq:IndexSetBound}.
\end{corollary}

\begin{remark}
The rate of convergence in the previous corollary 
is dependent on the dimension, but this dependency
is extremely mild as we have $\log\log d < 2$ for 
dimensions \(d\leq 1000\). Therefore, given that 
\(r\geq 2\), it still guarantees exponential 
convergence in \(q\).
\end{remark}

A straightforward consequence of the preceding theorem
is the following estimate on the interpolation error in 
case of algebraically increasing regions of analyticity.

\begin{corollary}\label{cor:interpest}
Let \(v\in C([0,1]^d)\) admit an analytic
extension \(\tilde{v}\) into the region 
\(\bs\Sigma({\bs\tau})\) for a monotonously increasing sequence
\(0<\tau_1\leq\ldots\leq\tau_d\le\dots\)
and \(\tau_k\geq ck^r\) for some \(c,r>1\).
Define
\(1<\rho_k\isdef2\tau_k+\sqrt{1+4\tau_k^{2}}\) for $k=1,\ldots,d$
and introduce the weights \(\omega_k\isdef\log\rho_k\).
Then, there holds for the interpolation 
\(I_{\boldsymbol{\omega},q}\colon C([0,1]^d)\to\Pcal_{\boldsymbol{\omega},q}\)
with Lebesgue constant \(L_{\boldsymbol{\omega},q}\)
that
\[
\big\|(\operatorname{Id}^{(d)} - I_{\boldsymbol{\omega},q})v\big\|_{C([0,1]^d)}
\leq(1+L_{\boldsymbol{\omega},q})c(r)2^{d+1} 
e^{-q(1-\frac 1 r\log\log d)}\|\tilde{v}\|_{C(\bs\Sigma({\bs\tau}))}
\]
with the constant \(c_r\) from \eqref{eq:IndexSetBound}.
\end{corollary}
\begin{proof}
Due to \(I_{\bs\omega,q}P_{\bs\omega,q}v=P_{\bs\omega,q}v\),
there holds by the triangle inequality
\begin{align*}
\big\|(\operatorname{Id}^{(d)} - I_{\boldsymbol{\omega},q})v\big\|_{C([0,1]^d)}
&\leq\big\|(\operatorname{Id}^{(d)} - 
P_{\boldsymbol{\omega},q})v\big\|_{C([0,1]^d)}+\big\|(P_{\boldsymbol{\omega},q} 
- I_{\boldsymbol{\omega},q})v\big\|_{C([0,1]^d)}\\
&\leq(1+L_{\boldsymbol{\omega},q})
\big\|(\operatorname{Id}^{(d)} - P_{\boldsymbol{\omega},q})v\big\|_{C([0,1]^d)}.
\end{align*}
The assertion follows now from an application of 
Corollary~\ref{cor:withConstants}.
\end{proof}

\section{The dimension weighted fast multipole method}
\label{sec:multipole}
In this section, we introduce the dimension weighted fast multipole method.
and provide the corresponding approximation error estimates.
\subsection{Clustering}
The first step is to provide a suitable hierarchical clustering
of the set of data sites \(X=\{{\bs x}_1,\ldots,{\bs x}_N\}\subset\Rbb^d\).
To this end, we introduce the concept of a \emph{cluster tree}.

\begin{definition}\label{def:cluster-tree}
Let $\mathcal{T}=(V,E)$ be a tree with vertices $V$ and edges $E$.
We define its set of leaves as
\(
\mathcal{L}(\mathcal{T})\isdef\{\nu\in V\colon\nu~\text{has no children}\}.
\)
The tree $\mathcal{T}$ is a \emph{cluster tree} for
the set $X=\{{\bs x}_1,\ldots,{\bs x}_N\}$, iff
$X$ is the {root} of $\mathcal{T}$ and
all $\nu\in V\setminus\mathcal{L}(\mathcal{T})$
are disjoint unions of their children.
For \(\nu\neq X\), we denote the parent cluster of \(\nu\)
by \(\operatorname{parent}(\nu)\).
The \emph{level} \(j_\nu\) of $\nu\in\mathcal{T}$ is its distance from
the root
and the \emph{bounding box} $B_{\nu}$ 
is the smallest axis-parallel cuboid that 
contains all points of \(\nu\).
The \emph{depth} of the cluster tree is given by 
\(J\isdef\max_{\nu\in\Tcal}j_\nu\). The set of clusters
on level $j$ as given by
\(
\mathcal{T}_j\isdef\{\nu\in\mathcal{T}\colon \nu~\text{has level}~j\}
\).
\end{definition}

Different possibilities exist for the branching degree within the
tree. Uniformly subdividing a given bounding box into \(2^d\)
congruent boxes yields a \(2^d\)-tree. For the sake of simplicity,
we focus here on binary trees, which are equivalent to \(2^d\) trees
by combining \(d\) successive levels to a single one.
In particular, we assume that two clusters on the same level have
approximately the same cardinality leading to the concept of
a \emph{balanced binary tree}.

\begin{definition}
Let $\Tcal$ be a cluster tree for $X$ with depth $J$. 
$\Tcal$ is called a \emph{balanced binary tree}, if all 
clusters $\nu$ satisfy the following conditions:
\begin{enumerate}
\item[(\emph{i})]
The cluster $\nu$ has exactly two children
if $j_{\nu} < J$. It has no sons if $j_{\nu} = J$.
\item[(\emph{ii})]
We have $|\nu|\sim 2^{J-j_{\nu}}$,
where $|\nu|$ denotes the
number of points contained in \({\nu}\).
\end{enumerate}
\end{definition}

A balanced binary tree can be constructed by \emph{cardinality 
balanced clustering}. This means that the root cluster 
is split into two child clusters of identical (or similar)
cardinality. This process is repeated recursively for the 
resulting child clusters until their cardinality falls below a 
certain threshold.
For the subdivision, the cluster's bounding box
is split along its longest edge such that the 
resulting two boxes both contain an equal number of points.
Thus, as the cluster cardinality halves with each level, 
we obtain $\mathcal{O}(\log d)$ levels in total. 
The total cost for constructing the cluster tree
is therefore $\mathcal{O}(N \log d)$.

\subsection{Farfield approximation}
Corollary~\ref{cor:anisoKernelAdm} suggests that the interaction
of points that are sufficiently distant can be approximated
by a polynomial without compromising the accuracy. This leads
to the concept of the \emph{admissibility condition}.

\begin{definition}
The clusters \(\nu,\nu'\in\Tcal\) with \(j_\nu=j_{\nu'}\) are called
\emph{admissible} iff
\begin{equation}\label{eq:admissibility}
\dist(B_\nu,B_\nu')\geq\eta
\max\big\{\operatorname{diam}B_\nu,\operatorname{diam}B_\nu'\big\}
\end{equation}
holds for some \(\eta>0\).

Letting 
\[\Tcal\boxtimes\Tcal\isdef
\{\nu\times\nu'\colon \nu,\nu'\in\Tcal,\ j_\nu=j_\nu'\}
\] denote the levelwise
Cartesian product of the cluster tree \(\Tcal\), 
the largest collection of 
admissible blocks \(\nu\times\nu'
\in\mathcal{T}\boxtimes\mathcal{T}\) such that 
\(\operatorname{parent}(\nu)\times\operatorname{parent}(\nu')\)
is called the \emph{farfield} \(\mathcal{F}\subset\Tcal\boxtimes\Tcal\).
The remaining non-admissible blocks correspond to the \emph{nearfield} 
\(\mathcal{N}\subset\mathcal{T}\boxtimes\mathcal{T}\).
\end{definition}

For the interpolation of the kernel, we consider the univariate 
Legendre polynomials \(\hat{p}_j
\), \(j=0,\ldots, q\), satisfying
\(\Pi_q=\spn\{\hat{p}_0,\ldots,\hat{p}_q\}\) 
such that
\[
\int_0^1 \hat{p}_i(t)\hat{p}_j(t)\d t=\delta_{i,j},
\]
as well as their tensorized version
\[
\hat{p}_{\bs\alpha}({\bs x})\isdef\hat{p}_{\alpha_1}(x_1)
\cdots \hat{p}_{\alpha_d}(x_n).
\]
In particular, there holds
\[
\int_{[0,1]^d}\hat{p}_{\bs\alpha}({\bs x})
\hat{p}_{\bs\alpha'}({\bs x})\d{\bs x}
=\delta_{\bs\alpha,\bs\alpha'}.
\]

Given an admissible cluster \(\nu\times\nu'\in\Fcal\), we now approximate
\begin{equation}\label{eq:KernelApprox}
\Kcal({\bs x},{\bs y})\approx\widetilde{\Kcal}({\bs x},{\bs y})\isdef
\sum_{{\bs\alpha},{\bs\alpha}'\in X_{{\bs\omega},q,d}}
c_{\bs\alpha,\bs\alpha'}^{\nu,\nu'}p_{\bs\alpha}^{\nu}({\bs x})p_{\bs\alpha'}^{\nu'}({\bs y})
\end{equation}
with coefficients \(c_{\bs\alpha,\bs\alpha'}^{\nu,\nu'}\in\Rbb\) and the
transported polynomials 
\[
p^\nu_{\bs\alpha}\isdef\hat{p}_{\bs\alpha}\circ a_{\nu}^{-1}.
\] 
Here, \(a_\nu\colon[0,1]^d\to\nu\) is the affine mapping
that maps the unit hypercube to a cluster \(\nu\in\Tcal\).

\begin{remark}
The approximation \eqref{eq:KernelApprox} amounts to a tensor product 
approximation between the two kernel components. The cost of this
expansion is, however, governed by the number of polynomials
\(n_X\isdef\#\Lambda_{{\bs\omega},q,d}\), i.e., the size of 
\({\bs C}_{\nu,\nu'}\in\Rbb^{n_X\times n_X}\). In particular, 
any sort of sparse approximation between the kernel components 
results in a sparsification of the matrix \({\bs C}_{\nu,\nu'}\) 
but does not impact the overall cost.
\end{remark}

To compute the coefficients \(c_{\bs\alpha,\bs\alpha'}^{\nu,\nu'}\) in \eqref{eq:KernelApprox}
for \({\bs\alpha},{\bs\alpha}'\in \Lambda_{{\bs\omega},q,d}\), we assume that there
exists a sequence of interpolation nodes \(\hat{\bs\xi}_{\bs\alpha}\in[0,1]^d\),
\({\bs\alpha}\in \Lambda_{{\bs\omega},q,d}\), such that the Vandermonde matrix
\begin{equation}\label{eq:Vandermonde}
{\bs V}\isdef\big[\hat{p}_{\bs\alpha'}(\hat{\bs\xi}_{\bs\alpha})\big]_{\bs\alpha,\bs\alpha'\in \Lambda_{{\bs\omega},q,d}}
\end{equation}
has full rank. Hence, we obtain the
coefficients \(c_{\bs\alpha,\bs\alpha'}^{\nu,\nu'}\) in matrix notation
according to
\[
{\bs C}_{\nu,\nu'}\isdef
[c_{\bs\alpha,\bs\alpha'}^{\nu,\nu'}]_{\bs\alpha,
\bs\alpha'\in \Lambda_{{\bs\omega},q,d}}
={\bs V}^{-1}{\bs S}_{\nu,\nu'}{\bs V}^{-\intercal}
\]
with the coupling matrix
\[
{\bs S}_{\nu,\nu'}\isdef\big[\Kcal\big(a_\nu(\hat{\bs\xi}_{\bs\alpha}),
a_\nu'(\hat{\bs\xi}_{\bs\alpha}')\big)\big]_{
\bs\alpha,\bs\alpha'\in \Lambda_{{\bs\omega},q,d}}.
\]
Next, we introduce the transported polynomials 
\(p^\nu_{\bs\alpha}\isdef\hat{p}_{\bs\alpha}\circ a_{\nu}^{-1}\)
and the cluster bases 
\[{\bs P}_{\nu}
\isdef[p^\nu_{\bs\alpha}({\bs x})]_{{\bs x}
\in\nu,{\bs\alpha}\in \Lambda_{{\bs\omega},q,d}}.
\]
Then, in view of \eqref{eq:KernelApprox}, we can approximate the matrix block
\({\bs K}_{\nu,\nu'}
\isdef[\Kcal({\bs x},{\bs x}')]_{{\bs x}\in\nu,{\bs x}'\in\nu'}\)
associated to the block-cluster \(\nu\times\nu'\in\Fcal\) according to
\begin{equation}\label{eq:LowRankBlock}
{\bs K}_{\nu,\nu'}\approx
\widetilde{\bs K}_{\nu,\nu'}={\bs P}_\nu
{\bs C}_{\nu,\nu'}{\bs P}_{\nu'}^\intercal
= {\bs P}_\nu{\bs V}^{-1}{\bs S}_{\nu,\nu'}
{\bs V}^{-\intercal}{\bs P}_{\nu'}^\intercal.
\end{equation}

By fixing the index set \(\Lambda_{{\bs\omega},q,d}\) for all clusters, 
we can easily construct nested cluster bases, see \cite{Borm} 
for example. Thus, for the transfer matrices 
\[
{\bs T}_{\nu_\text{child}}
\isdef{\bs V}^{-1}\big[p^\nu_{{\bs\alpha}'}\big(a_{\nu_\text{child}}
(\hat{\bs\xi}_{\bs\alpha})\big)\big]_{{\bs\alpha},{\bs\alpha}'
\in \Lambda_{{\bs\omega},q,d}},
\]
there holds the two-scale relation
\[
{\bs P}_\nu=\begin{bmatrix}{\bs P}_{\nu_{\text{child}_1}}
{\bs T}_{\nu_{\text{child}_1}}\\{\bs P}_{\nu_{\text{child}_2}}
{\bs T}_{\nu_{\text{child}_2}}\end{bmatrix}.
\]

\subsection{Error analysis}
Based on the degenerate kernel approximation
\eqref{eq:KernelApprox} and Theorem~\ref{th:errorEstimate},
the error analysis of the dimension weighted fast multipole
method can now be performed along the lines of \cite{Borm}.
For the reader's convenience, we recall the error estimate below.
Given the set \(X=\{{\bs x}_1,\ldots,{\bs x}_N\}\subset\Rbb^d\),
we define the function space \(\Hcal_X\isdef
\operatorname{span}\{\varphi_1,\ldots,\varphi_N\}\),
where \(\varphi_i\isdef\kernel({\bs x}_i,\cdot)\in\Hcal\).
\begin{lemma}\label{lem:erradm} Let the assumptions of 
Corollary~\ref{cor:interpest}
be satisfied for \(\tau_k\in[0,c_k\eta\rho/b_k]\).
Then, for any \(\nu\times\nu'\in\Fcal\), there holds
\[
\big\|{\bs K}_{\nu,\nu'}-
\widetilde{\bs K}_{\nu,\nu'}\big\|_F\leq
2(1+L_{\boldsymbol{\omega},q})c_r2^d e^{-q(1-\frac 1 r\log\log d)}
\frac{c_{\kernel}}{1-\gamma}\sqrt{|\nu||\nu'|}.
\]
\end{lemma}

\begin{proof}
There holds by Corollary~\ref{cor:interpest} that
\begin{align*}
\big\|{\bs K}_{\nu,\nu'}-
\widetilde{\bs K}_{\nu,\nu'}\big\|_F^2
&=\sum_{{\bs x}\in\nu,{\bs x}'\in\nu'}
\big(\kernel({\bs x},{\bs x}')-\widetilde{\kernel}({\bs x},{\bs x}')\big)^2\\
&\leq\sum_{{\bs x}\in\nu,{\bs x}'\in\nu'}
\bigg(2(1+L_{\boldsymbol{\omega},q})c_r2^d e^{-q(1-\frac 1 r\log\log d)}
\frac{c_{\kernel}}{1-\gamma}\bigg)^2\\
&=\bigg(2(1+L_{\boldsymbol{\omega},q})c_r2^d e^{-q(1-\frac 1 r\log\log d)}
\frac{c_{\kernel}}{1-\gamma}\bigg)^2|\nu||\nu'|.
\end{align*}
Taking square roots on both sides yields the assertion.
\end{proof}

The following result is immediately obtained by summing up the interpolation
errors over all farfield blocks, taking into account that we have a panelization 
of the matrix into disjoint submatrices.

\begin{corollary}\label{th:Frobtot} Let the assumptions of Corollary~\ref{cor:interpest}
be satisfied for \(\tau_k\in[0,c_k\eta\rho/b_k]\).
Then, there holds
\[
\big\|{\bs K}-
\widetilde{\bs K}\big\|_F\lesssim
(1+L_{\boldsymbol{\omega},q})c_r2^d e^{-q(1-\frac 1 r\log\log d)}
\frac{c_{\kernel}}{1-\gamma}N.
\]
\end{corollary}

For quasi-uniform sets \(X\), we have the following 
refinement of Corollary~\ref{th:Frobtot}.

\begin{corollary}
Let the assumptions of Corollary~\ref{cor:interpest}
be satisfied for \(\tau_k\in[0,c_k\eta\rho/b_k]\). In case of 
quasi-uniform points, the matrix $\widetilde{\bs K}$ satisfies the 
following error estimate
\[
\frac{\big\|{\bs K}-
\widetilde{\bs K}\big\|_F}{\big\|
\bs K\big\|_F}\lesssim(1+L_{\boldsymbol{\omega},q})c_r
2^d e^{-q(1-\frac 1 r\log\log d)}
\frac{c_{\kernel}}{1-\gamma}.
\]
\end{corollary}
\begin{proof}
From Theorem~\ref{th:Frobtot} we know that, for
quasi-uniform points \(X=\{{\bs x}_1,\ldots,{\bs x}_N\}\subset\Rbb^d\) 
as in Definition~\ref{def:qunifpoints}, there holds
\[
\frac{1}{N^2}\big\|\bs K\big\|_F^2=\frac{1}{N^2}
\sum_{i=1}^{N}\sum_{j=1}^N\big|\kernel(\bs{x}_i,\bs {x}_j)\big|^2
\sim \int_{\Omega}\int_{\Omega}\big|\kernel(\bs{x},\bs{y})\big|^2\d 
\bs{x}\d \bs{y},
\]
and so \(\big\|\bs K\big\|_F \sim N\). Exploiting this, 
we obtain that the matrix \(\widetilde{\bs K}\) satisfies
\[
\frac{\big\|{\bs K}-
\widetilde{\bs K}\big\|_F}{\big\|
\bs K\big\|_F}\lesssim(1+L_{\boldsymbol{\omega},q})
c_r2^d e^{-q(1-\frac 1 r\log\log d)}\frac{c_{\kernel}}{1-\gamma},
\]
as claimed.
\end{proof}

\subsection{Approximate Fekete points}
We now introduce the concept of approximate Fekete points, see \cite{BDMSV10},
for the interpolation of the kernel matrix \(\bs K\). For this purpose, 
we recall the definition of Fekete points and the Lebesgue constant, where the 
latter is related to the Vandermonde matrix \(\bs V\)
from \eqref{eq:Vandermonde}.
Let \(\{\hat{\bs\xi}_{\bs\alpha}\}\subset[0,1]^d\) and \(\{\hat{p}_{\bs\alpha}
({\bs x})\}\), \({\bs\alpha}\in \Lambda_{{\bs\omega},q,d}\), be a set of 
interpolation nodes and linearly independent continuous functions,
respectively, such that the set of interpolation nodes is unisolvent 
in \(\spn\{\hat{p}_{\bs\alpha}:{\bs\alpha \in \Lambda_{{\bs\omega},q,d}}\}\). 
Note that, in our setting, \(\{\hat{p}_{\bs\alpha}\}\) are simply 
the tensorized Legendre polynomials. If the set of nodes 
\(\hat{\bs\xi}_{\bs\alpha}\) are chosen to maximize in 
\([0,1]^d\) the modulus of the determinant 
$\det({\bs V})$ of the Vandermonde matrix, then 
\(\hat{\bs\xi}_{\bs\alpha}\) are called Fekete points. 
Denoting by ${\bs V}_j({\bs x})$ the Vandermonde 
matrix ${\bs V}$ with the $j$-th column being replaced by the 
vector $\big[\hat{p}_{\bs\alpha'}(\bs x)\big]_{\bs\alpha'
\in \Lambda_{{\bs\omega},q,d}}$, then the Lagrange basis
\(\ell_j({\bs x})\isdef\det\big(\bs V_j({\bs x})\big)
/\det({\bs V})\),
\(j=,1\ldots,|\Lambda_{{\bs\omega},q,d}|\), obviously satisfies
\[
\|\ell_j\|_{\infty} \leq 1.
\]
In particular, the Lebesgue constant
\[
L_{\boldsymbol{\omega},q}= \max\limits_{\bs x \in [0,1]^d} 
\sum_{j=1}^{|L_{{\bs\omega},q,d}|} \big| \ell_j(\bs x)\big|
\]
is such that \(L_{\boldsymbol{\omega},q}\leq |\Lambda_{{\bs\omega},q,d}|\).

Fekete points are known to be ``good'' interpolation points. However, 
their computation becomes non-feasible very rapidly and as such they 
can be solved numerically in very special cases only. Therefore, a greedy 
algorithm has been proposed in \cite{BDME16, BDMSV11, BDMSV10} for the 
computation of approximate Fekete points, which are good approximations 
to the true ones. These are realized as a column-pivoted QR decomposition 
of the Vandermonde matrix based on a candidate points from the
Halton sequence, see \cite{Hal60}. For all the details
on approximate Fekete points, we refer the reader to \cite{BDME16} 
and \cite{BDMSV10}.

\section{Numerical experiments}\label{sec:numerix}
\subsection{Comparison between interpolation techniques}
We shall present a comparison between tensor product 
interpolation (TPI), total degree interpolation (TDI), 
and weighted total degree interpolation (WTDI). To this end, 
we compare the number of points used by the three methods in 
relation to the dimension of the domain under consideration 
and the number of polynomials. The way of defining the weights 
in case of WTDI follows Theorem~\ref{th:errorEstimate}, where
we have set $\tau_k = k^{2}$, $k\in\Nbb$.

It is well known that, given the dimension $d$ and the polynomial 
degree $q$, the number of points for TPI is
\[
|\{{\bs x}^{\bs\alpha}:\|\bs\alpha\|_\infty\leq q\}|=(q+1)^{d}
\]
while 
for TDI it is 
\[
|\Lambda_{{\bs 1},q,d}|=\binom{q+d}{d}.
\]
As we can see in the left and
middle plot of Figure~\ref{fig:comparison_number_points}, the 
difference in the number of points used by the different techniques 
is evident. For WTDI, we can observe in the right plot of 
Figure~\ref{fig:comparison_number_points} that the required 
number of points is much smaller and bounded by the worst-case
scenario which is obtained for dimension \(d=20\) and polynomial 
degree \(q=10\), where the number of points used is \(244\).

\begin{figure}[htb]
\begin{center}
\scalebox{0.64}{
\begin{tikzpicture}
\begin{axis}[zmode=log,
       zmax = 1e21,
       zmin = 1e0,
       title = Number of points for TPI,
y dir=reverse, 
grid,
ytick={0,5,10,15,20},
xtick={0,2,4,6,8,10},
y dir=reverse, 
y label style={inner sep=5pt, anchor=center },
xlabel=Pol. degree,
y label style={rotate=75},
grid=both,
grid style={line width=10pt, gray!10},
grid style={ultra thin, gray!50},
ylabel=Dimension,
x label style= {rotate=90},
  every axis x label/.append style={sloped, at={(rel axis cs: 0.5, -0.175, 0)}, below},
  every axis y label/.append style={sloped, at={(rel axis cs: -0.175, 0.5, 0)}, below}, 
colormap/Paired,
view={-75}{20},
height=6.85cm,width=6.85cm]
\addplot3[only marks,scatter] file {test_TPI_2.txt};
     \end{axis}
\end{tikzpicture}}
\scalebox{0.64}{
\begin{tikzpicture}
\begin{axis}[zmode=log,
       zmax = 1e21,
       zmin = 1e0,
       title = Number of points for TDI,
y dir=reverse, 
grid,
ytick={0,5,10,15,20},
xtick={0,2,4,6,8,10},
y dir=reverse, 
y label style={inner sep=5pt, anchor=center },
xlabel=Pol. degree,
y label style={rotate=75},
grid=both,
grid style={line width=10pt, gray!10},
grid style={ultra thin, gray!50},
ylabel=Dimension,
x label style= {rotate=90},
  every axis x label/.append style={sloped, at={(rel axis cs: 0.5, -0.175, 0)}, below},
  every axis y label/.append style={sloped, at={(rel axis cs: -0.175, 0.5, 0)}, below}, 
colormap/Paired,
view={-75}{20},
height=6.85cm,width=6.85cm]
\addplot3[only marks,scatter] file {test_TDI_2.txt};
\end{axis}
\end{tikzpicture}}
\scalebox{0.64}{
\begin{tikzpicture}
\begin{axis}[zmode=log,
       zmax = 1e21,
       zmin = 1e0,
       title = Number of points for WTDI,
y dir=reverse, 
grid,
ytick={0,5,10,15,20},
xtick={0,2,4,6,8,10},
y label style={inner sep=10pt, anchor=center },
xlabel=Pol. degree,
y label style={rotate=75},
grid=both,
grid style={line width=10pt, gray!10},
grid style={ultra thin, gray!50},
ylabel=Dimension,
x label style= {rotate=90},
  every axis x label/.append style={sloped, at={(rel axis cs: 0.5, -0.175, 0)}, below},
  every axis y label/.append style={sloped, at={(rel axis cs: -0.175, 0.5, 0)}, below}, 
colormap/Paired,
view={-75}{20},
height=6.85cm,width=6.85cm]
\addplot3[only marks,scatter] file {test_WTDI_2.txt};
\end{axis}
\end{tikzpicture}}
\caption{\label{fig:comparison_number_points}
Number of interpolation points required for TDI, TPI, 
and WTDI based on a quadratic increase of the sequence
 \(\{\tau_k\}\) of the convergence radii.}
\end{center}
\end{figure}
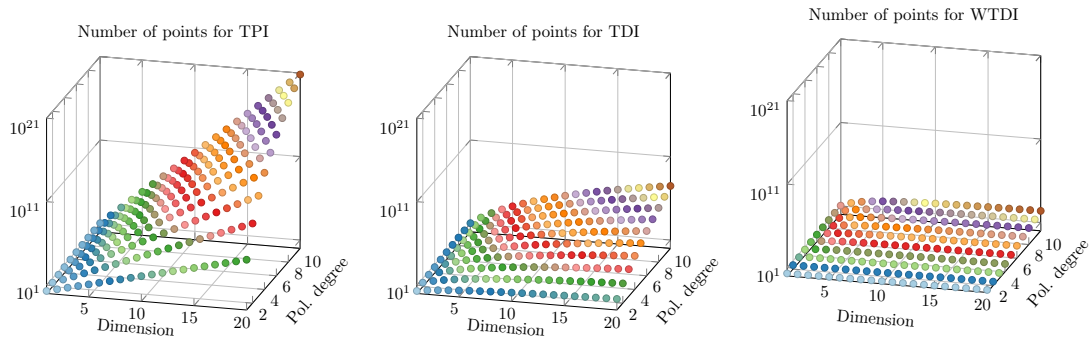

In the next experiment, we study how the number of points
changes when the sequence $\{\tau_k\}$ increases at higher rates. 
Namely, in Figure~\ref{fig:number_points_WTDI}, the numbers 
of interpolation points is depicted which are obtained for 
$\tau_k = k^r$ for $r=2,3,4$. As we can expect, the higher 
the value of $r$, the fewer points used as a consequence 
of the definition of the weight vector \(\boldsymbol{\omega}\) 
and the set \(\Lambda_{\boldsymbol{\omega},q,k}\). 

\begin{figure}[htb]
\pgfplotsset{width=7cm,compat=1.3} 
\centering
\begin{tikzpicture}
    \pgfplotsset{compat=1.3}
        \begin{axis}[zmode=log,
y dir=reverse, 
y label style={inner sep=5pt, anchor=center },
ylabel=Dimension,
y label style={rotate=75},
ztick={1e0,1e1,1e2,1e3},
grid=both,
ytick={0,5,10,15,20},
xtick={0,2,4,6,8,10},
zmax = 2.7e2,
grid style={line width=10pt, gray!10},
grid style={ultra thin, gray!50},
xlabel=Pol. degree,
 legend style = {
        legend cell align=left,
        mark options={scale=3},
        /tikz/every even column/.append style={column sep=1cm}
    },
x label style= {rotate=90},
  every axis x label/.append style={sloped, at={(rel axis cs: 0.5, -0.175, 0)}, below},
  every axis y label/.append style={sloped, at={(rel axis cs: -0.175, 0.5, 0)}, below}, 
legend style={font=\small,cells={anchor=east},
        legend pos=outer north east},
view={-55}{30},
height=6.85cm,width=6.85cm]
\addplot3[only marks,  mark = diamond,scatter, mark size=0.9pt, colormap/violet] file {test_WTDI_2.txt};
\addlegendentry{$r=2$}
\addplot3[only marks, mark size=0.9pt, mark = square,scatter, colormap/cool] file {test_WTDI_3.txt};
\addlegendentry{$r=3$}
\addplot3[only marks, mark size=0.9pt,mark = star,scatter, colormap/greenyellow] file {test_WTDI_4.txt};
\addlegendentry{$r=4$}
\end{axis};
\end{tikzpicture}
\caption{\label{fig:number_points_WTDI}
Comparison of the number of points used by WTDI based on 
the different rates of increase $\tau_k = k^{r}$ for \(r=2,3,4\).}
\end{figure}
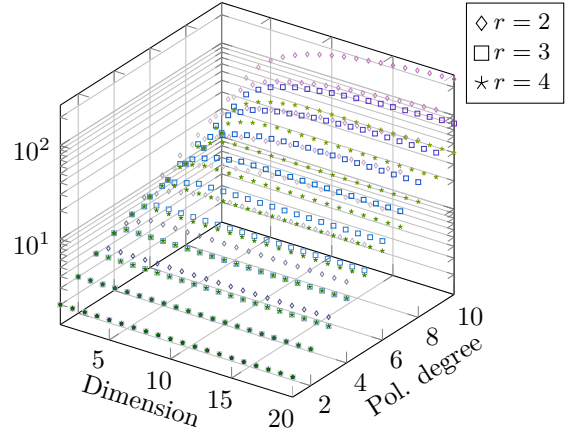

\subsection{Accuracy of the kernel interpolation}
To benchmark the accuracy of the kernel interpolation,
we consider \(N=100\,000\) uniformly distributed random samples 
$\{{\bs x}_1,\ldots{\bs x}_N\}$ in the hypercube \([0,1]^{20}\),
which is rescaled as in Section~\ref{sec:problem} by 
\(b_k=k^{-r}\) for $r=2,3,4$. We 
intend to evaluate the performance of the kernel interpolation by 
employing a validation approach that involves the splitting of the data 
into distinct training and test sets. The training dataset comprises the 
aforementioned \(N\) uniformly distributed random samples, utilized 
for solving the system \eqref{eq:LGS} of linear equations, while 
the test set consists of \(10\,000\) points. 

To this end, we choose the exponential kernel $\kernel(\bs{x},\bs{y})
= \exp(\|\bs{x}-\bs{y}\|_2/\sigma)$, where we consider \(15\) different values 
of the length scale parameter $\sigma$. The values are derived 
from a logarithmically equispaced grid within the interval 
$\big[\max\{10^{-5}, q_{X}\},\operatorname{diam}(B_X)\big]$. 
Here, $\operatorname{diam}(B_X)$ is the diameter of the root 
cluster $B_X$ of the cluster tree and $q_X$ is the separation 
radius from Definition~\ref{def:qunifpoints}. 
The right-hand side of \eqref{eq:LGS} is calculated as 
\[ 
  y_i\isdef \frac{\operatorname{sin}(4 \pi \| \bs x_i \|_2)}
    {8\pi \|\bs x_i \|_2},
\]
where \(\bs x_i \in X\). For the ridge parameters \(\lambda\)
considered for regularizing the underlying system \eqref{eq:LGS} 
of linear equations, we consider \(15\) different values, chosen 
logarithmically equispaced grid from the interval 
$\big[10^{-6}, 10^{-1}\big]$. For the multipole method, we use 
\(\eta=0.5\) for the construction of the cluster tree. In order 
to define the weights $w_k$ for the polynomial interpolation, 
we set $\tau_k^{-1} = \eta b_k$ in each particular case of $r$ 
and define \(\omega_k = \log \rho_k \), where 
\[ 
\rho_k= e \cdot \frac{2\tau_k + \sqrt{1+4\tau_k^2}}
    {2\tau_1+\sqrt{1+4\tau_1^2}}. 
\]
This choice ensures the normalization of the weights such that $\omega_1 = 1$.
Moreover, we set the maximum degree to \(q=8\) 
for the weighted total degree index set \(\Lambda_{\boldsymbol{\omega},q,d}\). 
The system \eqref{eq:LGS} of linear equations is solved by 
using the conjugate gradient method with a relative error 
less than \(10^{-6}\).

For each value of \(r\in\{2,3,4\}\), we report the average
prediction error (PE) and the compression error (CE). The 
PE is the relative error, measured in the Euclidean norm,
of the predicted values and the true values at \(1\,000\)
randomly chosen points in \(B_X\) that are not contained 
in \(X\). The compression error is quantified by the 
relative error between \(100\) randomly selected columns 
of \({\bs K}\) and the corresponding columns of 
\(\widetilde{\bs K}\), measured in the Frobenius norm. 
To obtain a second order statistics for PE and CE,
both errors are computed five times. For PE, we only
report the corresponding average, while we also show
the standard deviation in case of CE. The results for 
\(r=2,3,4\) are shown in Figures \ref{fig:20_2_test}, 
\ref{fig:20_3_test}, and \ref{fig:20_4_test}, respectively. 

\begin{figure}[htb]
\begin{center}
\pgfplotsset{width=7cm,compat=1.3}
\begin{tikzpicture}
    \pgfplotsset{small}
    \matrix {
\begin{axis}[zmode=log,
xmode = log,
ymode = log,
       zmax = 1,
       zmin = 1e-3,
grid = both,
ticklabel style = {font=\tiny},
xtick={1e-3,1e-2,1e-1, 1},
y label style={ anchor=center },
xlabel=Length scale,
y label style={rotate=75},
grid style={ultra thin, gray!20},
ylabel=Ridge parameter,
x label style= {rotate=90},
  every axis x label/.append style={sloped, at={(rel axis cs: 0.5, -0.250, 0)}, below},
  every axis y label/.append style={sloped, at={(rel axis cs: 1.250, 0.5, 0)}, below}, 
colormap/Paired,
shader=interp,
y dir = reverse,
view={45}{30},
mesh/cols=15,
mesh/ordering=y varies,
height=6.85cm,width=6.85cm]
 \addplot3+[mesh, scatter] file{20_2_APE.txt};
`\end{axis}
& 
\begin{axis}[
xmode = log,
ymode = log,
ymin = 1e-17,
grid,
xtick={1e-4,1e-3,1e-2,1e-1, 1},
y label style={ anchor=center },
xlabel=Length scale,
y label style={rotate=75},
grid=both,
grid style={ultra thin, gray!20},
height=6.85cm,width=6.85cm]
 \addplot[blue,mark size=1.25pt, error bars/.cd, y dir=both,y explicit, error bar style={red}]
coordinates {

   (7.7203767e-03,1.5593271e-09) +- (0,9.6947609e-11)
   (1.0429075e-02,2.0113814e-07) +- (0,1.0046240e-08)
   (1.4088122e-02,5.5834200e-06) +- (0,1.8263452e-07)
   (1.9030947e-02,5.1738000e-05) +- (0,2.0025457e-06)
   (2.5707965e-02,2.4447743e-04) +- (0, 7.0773429e-06)
   (3.4727619e-02,7.1811166e-04) +- (0,3.9343478e-05)
   (4.6911823e-02,1.4450585e-03) +- (0,5.5312542e-05)
   (6.3370862e-02,2.1626519e-03) +- (0,4.8804767e-05)
   (8.5604564e-02,2.6918346e-03) +- (0,9.6068232e-05)
   (1.1563897e-01,2.7497235e-03) +- (0,8.3154115e-05)
   (1.5621097e-01,2.5289327e-03) +- (0,7.8766469e-05)
   (2.1101767e-01,2.2680772e-03) +- (0,1.4854392e-04)
   (2.8505334e-01,1.7990273e-03) +- (0,5.3940832e-05)
   (3.8506446e-01,1.4447849e-03) +- (0,4.1024760e-05)
   (5.2016454e-01,1.0838726e-03) +- (0,3.5577749e-05)
   
};
 \end{axis}
\\ };
\end{tikzpicture}
 \caption{\label{fig:20_2_test}
 Average prediction error (left) and average compression
 error (right) for the benchmark study in case of $r=2$.}
\end{center}
\end{figure}
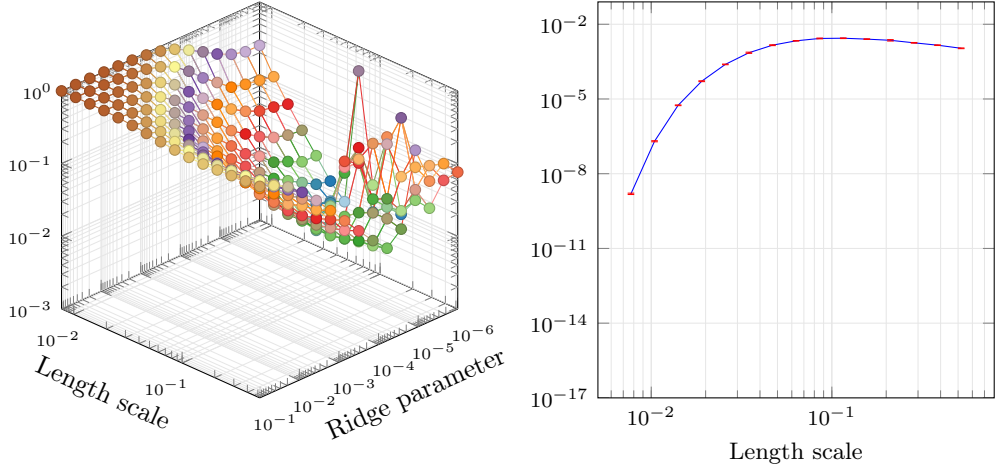

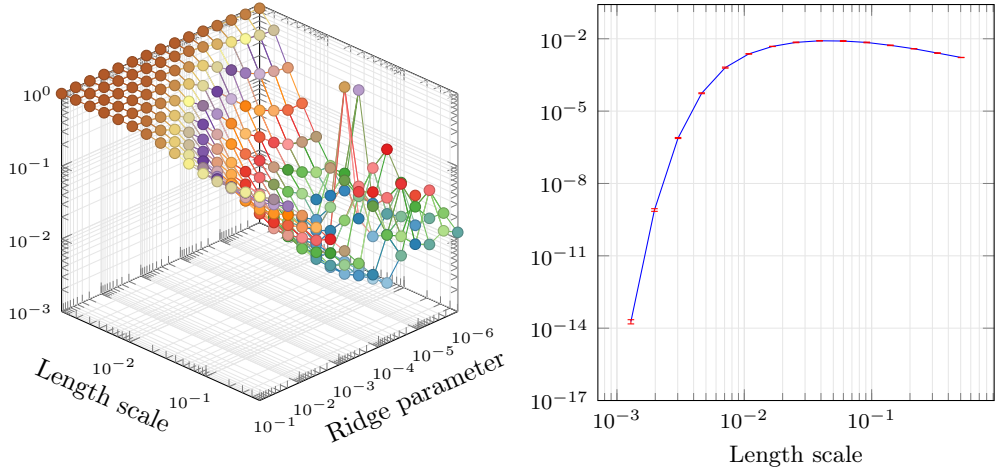
\begin{figure}[htb]
\begin{center}
\pgfplotsset{width=7cm,compat=1.3}
\begin{tikzpicture}
    \pgfplotsset{small}
    \matrix {
\begin{axis}[zmode=log,
xmode = log,
ymode = log,
       zmax = 1,
       zmin = 1e-3,
grid = both,
ticklabel style = {font=\tiny},
xtick={1e-3,1e-2,1e-1, 1},
y label style={ anchor=center },
xlabel=Length scale,
y label style={rotate=75},
grid style={ultra thin, gray!20},
ylabel=Ridge parameter,
x label style= {rotate=90},
  every axis x label/.append style={sloped, at={(rel axis cs: 0.5, -0.250, 0)}, below},
  every axis y label/.append style={sloped, at={(rel axis cs: 1.250, 0.5, 0)}, below}, 
colormap/Paired,
shader=interp,
y dir = reverse,
view={45}{30},
mesh/cols=15,
mesh/ordering=y varies,
height=6.85cm,width=6.85cm]
 \addplot3+[mesh, scatter] file{20_3_APE.txt};
\end{axis}
& 
\begin{axis}[
xmode = log,
ymode = log,
ymin = 1e-17,
grid,
xtick={1e-4,1e-3,1e-2,1e-1, 1},
y label style={ anchor=center },
xlabel=Length scale,
y label style={rotate=75},
grid=both,
grid style={ultra thin, gray!20},
height=6.85cm,width=6.85cm]
 \addplot[blue,mark size=1.25pt, error bars/.cd, y dir=both,y explicit, error bar style={red}]
coordinates {

   (1.2867724e-03,   1.8730648e-14) +- (0,3.7205385e-15)
   (1.9711989e-03,   7.9884791e-10) +- (0,1.0031807e-10)
   (3.0196677e-03,   7.6721484e-07) +- (0,3.9415700e-08)
   (4.6258108e-03,   5.5011740e-05) +- (0,2.3213683e-06)
   (7.0862518e-03,   6.3687102e-04) +- (0,3.3157338e-05)
   (1.0855387e-02,   2.3584685e-03) +- (0,6.3349271e-05)
   (1.6629301e-02,   4.8145409e-03) +- (0,7.1090907e-05)
   (2.5474327e-02,   7.0903311e-03) +- (0,2.3051539e-04)
   (3.9023968e-02,   8.2358550e-03) +- (0,2.0640436e-04)
   (5.9780583e-02,   8.0561111e-03) +- (0,3.1503882e-04)
   (9.1577515e-02,   7.0332478e-03) +- (0,2.3548773e-04)
   (1.4028705e-01,   5.3429296e-03) +- (0,1.5615353e-04)
   (2.1490488e-01,   3.8530075e-03) +- (0,4.3128082e-05)
   (3.2921150e-01,   2.5513336e-03) +- (0,7.2639838e-05)
   (5.0431711e-01,   1.6734688e-03) +- (0,2.2064662e-05)
};
 \end{axis}
\\ };
\end{tikzpicture}
 \caption{\label{fig:20_3_test}
 Average prediction error (left) and average compression
 error (right) for the benchmark study in case of $r=3$.}
\end{center}
\end{figure}

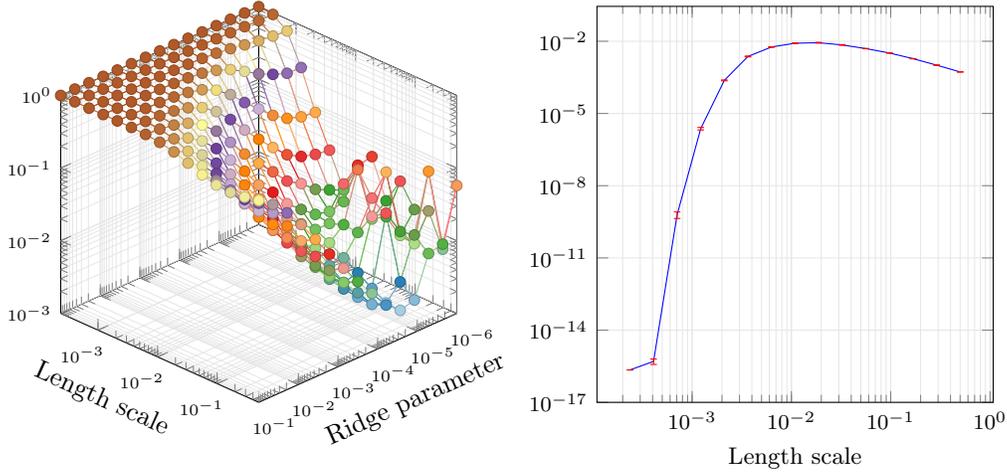
\begin{figure}[htb]
\begin{center}
\pgfplotsset{width=7cm,compat=1.3}
\begin{tikzpicture}
    \pgfplotsset{small}
    \matrix {
\begin{axis}[zmode=log,
xmode = log,
ymode = log,
       zmax = 1,
       zmin = 1e-3,
grid = both,
ticklabel style = {font=\tiny},
xtick={1e-3,1e-2,1e-1, 1},
y label style={ anchor=center },
xlabel=Length scale,
y label style={rotate=75},
grid style={ultra thin, gray!20},
ylabel=Ridge parameter,
x label style= {rotate=90},
  every axis x label/.append style={sloped, at={(rel axis cs: 0.5, -0.250, 0)}, below},
  every axis y label/.append style={sloped, at={(rel axis cs: 1.250, 0.5, 0)}, below}, 
colormap/Paired,
shader=interp,
y dir = reverse,
view={45}{30},
mesh/cols=15,
mesh/ordering=y varies,
height=6.85cm,width=6.85cm]
 \addplot3+[mesh, scatter] file{20_4_APE.txt};
\end{axis}
& 
\begin{axis}[
xmode = log,
ymode = log,
ymin = 1e-17,
grid,
xtick={1e-4,1e-3,1e-2,1e-1, 1},
y label style={ anchor=center },
xlabel=Length scale,
y label style={rotate=75},
grid=both,
grid style={ultra thin, gray!20},
height=6.85cm,width=6.85cm]
 \addplot[blue,mark size=1.25pt, error bars/.cd, y dir=both,y explicit, error bar style={red}]
coordinates { 

   (2.3662863e-04,   2.2204460e-16) +- (0,0.0000000e+00)
   (4.0890597e-04,  5.0209346e-16) +- (0,1.2686991e-16)
   (7.0660974e-04,   6.2400326e-10) +- (0,1.8935204e-10)
   (1.2210566e-03,   2.3512270e-06) +- (0,2.8076482e-07)
   (2.1100461e-03,   2.3732192e-04) +- (0,9.5427470e-06)
   (3.6462640e-03,   2.3342892e-03) +- (0,9.6456638e-05)
   (6.3009245e-03,   5.7848173e-03) +- (0,1.7489419e-04)
   (1.0888309e-02,   8.3022076e-03) +- (0,2.0008730e-04)
   (1.8815536e-02,   8.6900677e-03) +- (0,3.0870819e-04)
   (3.2514177e-02,   7.0478058e-03) +- (0,2.2092410e-04)
   (5.6186106e-02,   4.9619731e-03) +- (0,6.6517495e-05)
   (9.7092369e-02,   3.2530630e-03) +- (0,4.2815937e-05)
   (1.6778041e-01,   1.8934488e-03) +- (0,2.7416321e-05)
   (2.8993285e-01,   1.0187821e-03) +- (0,2.4449106e-05)
   (5.0101830e-01,   5.3138075e-04) +- (0,1.1991839e-05)

};
 \end{axis}
\\ };
\end{tikzpicture}
 \caption{\label{fig:20_4_test}
 Average prediction error (left) and average compression
 error (right) for the benchmark study in case of $r=4$.}
\end{center}
\end{figure}

The left-hand side of each figure shows the average 
PE, while the right-hand-side shows the average CE.
The \(x\)-axis in the PE plot corresponds to the
length scale parameter \(\sigma\), while the \(y\)-axis
corresponds to the ridge parameter. In the CE plots, the 
\(x\)-axis corresponds to the length scale parameter 
\(\sigma\) and the bars indicate the standard deviation.
The best observed parameter combination for \(r=2\)
is $\sigma=4.69\cdot{10^{-2}}$ and $\lambda=10^{-6}$, 
resulting in an error of \(5.98\cdot 10^{-3}\pm 
3.94\cdot 10^{-4}\). For \(r=3\), we obtain $\sigma=3.29
\cdot{10^{-1}}$ and $\lambda=6.11\cdot 10^{-5}$, which 
amounts to an error of \(6.31\cdot 10^{-3}\pm 3.24\cdot 
10^{-4}\), while \(r=4\) yields $\sigma=5.01\cdot{10^{-1}}$
and $\lambda=2.68\cdot 10^{-5}$ with an error of 
\(2.46\cdot 10^{-3}\pm 1.38\cdot 10^{-4}\).

The compression error shows a similar behavior in all three
cases. It is very small for low values of the length scale
parameter. This is explained by the fact that the kernel
matrix gets close to a diagonal matrix in this case, resulting
in a farfield that is basically zero.
Then the compression error monotonically increases up to a
a maximum and then monotonically decreases. This effect
is explained by the fact that at some point the kernel
gets very smooth apart from the diagonal and is well
approximated by the polynomials on the farfield.

\subsection{Shape uncertainty quantification}
As a relevant application for kernel interpolation, we consider 
the shape uncertainty quantification model introduced in 
\cite{HPS16}. Concretely, we consider 
the stationary diffusion problem
\begin{equation*}
  -\Delta u(\omega) = f\ \text{in $D(\omega)$}, \quad
  	u(\omega) = 0\ \text{on $\partial D(\omega), \quad \omega\in\Omega$}
\end{equation*}
for an uncertain domain \(D(\omega)\subset\mathbb{R}^n\). To 
model the shape uncertainty, we introduce a random deformation 
field \(\bs\chi\colon D_0\times\Omega\to\Rbb^n\) with
\begin{equation*}
\|{\bs\chi}(\omega)\|_{C^1(\overline{D_0};\mathbb{R}^n)},
\|{\bs\chi}^{-1}(\omega)\|_{C^1(\overline{D(\omega)};\mathbb{R}^n)}
\lesssim 1
\quad\text{for \(\mathbb{P}\)-a.e.\ }\omega\in\Omega,
\end{equation*}
which satisfies
\[
D(\omega)=\bs\chi(D_0,\omega).
\]
The deformation field is represented by
the Karhunen-Lo\`eve expansion
\[
{\bs\chi}(\omega)=\operatorname{Id}+
\sum_{k=1}^M\sigma_k{\bs\chi}_k{Y}_k(\omega),
\quad\omega\in\Omega,
\]
where we assume that the random variables \(\{Y_k\}_k\) are 
independent and identically uniformly distributed on \([-1,1]\). 
Moreover, the vector fields \({\bs\chi}_k\) amount to the eigenfunctions
of the covariance kernel associated to \({\bs\chi}\) and \(\sigma_k\) 
the corresponding singular values. For all the details, we 
refer to \cite{HPS16}.

By parametrizing the random variables by their image, we obtain 
the parametric deformation field
\begin{equation}\label{eq:pKL}
{\bs\chi}({\bs\zeta})=\operatorname{Id}+
\sum_{k=1}^M\sigma_k{\bs\chi}_k\zeta_k,
\quad
{\bs\zeta}\in[-1,1]^M,
\end{equation}
which in turn yields the parametric diffusion problem
\begin{equation}\label{SPDE}
  -\Delta u({\bs\zeta}) = f\ \text{in $D({\bs\zeta})$}, \quad
  	u({\bs\zeta}) = 0\ \text{on $\partial D({\bs\zeta}), 
    \quad {\bs\zeta}\in[-1,1]^M$}.
\end{equation}
As quantity of interest, we consider the 
linear output functional
\begin{equation}\label{QoI}
F(u)({\bs\zeta})=\big(u\circ{\bs\chi}^{-1}
({\bs\zeta}),\phi\big)_{L^2(D_0)}
\end{equation}
with a square integrable function \(\phi\colon D_0\to\Rbb\).

It is well-known that the parametric regularity of the
solution \(u\) of \eqref{SPDE} depends,
amongst others, on the spatial regularity of \(f\).
Namely, it has been shown in \cite{HPS16} that the decay in 
the sequence $\{\sigma_k\}$ of the singular values of the series expansion \eqref{eq:pKL}
propagates through the diffusion problem: The solution's derivatives 
satisfy the estimate
\[
\big\|\partial_{\bs\zeta}^{\bs\alpha} u({\bs\zeta})\big\|_{H^1(D_0)} 
    \lesssim |\bs\alpha|!\,\bs\gamma^{\bs\alpha}
\]
with the weights $\gamma_k\lesssim\sigma_k 
\|{\bs\chi}_k\|_{W^{1,\infty}(D_0)}$, $k\in\mathbb{N}$, 
provided that the right-hand side $f$ is smooth. Hence, 
as the quantity of interest is linear in $u$, the weight 
$\gamma_k$ with respect to the $k$-th input parameter 
is retained. Since, however, collocation approaches 
for the approximation of \(F(u)({\bs\zeta})\) are not 
efficient for non-smooth source terms \(f\), due to 
the lack of smoothness, we consider here the fast 
anisotropic kernel method presented in this article.

\begin{figure}[htb]
\includegraphics[scale=0.2,clip, trim= 230 60 80 40]{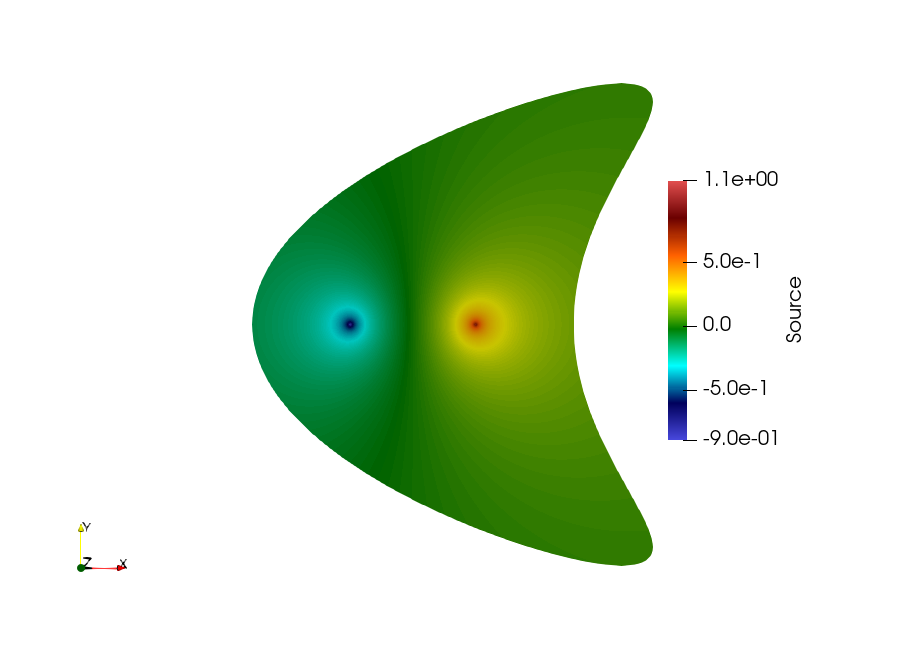}
\qquad\qquad
\includegraphics[scale=0.2,clip, trim= 230 60 80 40]{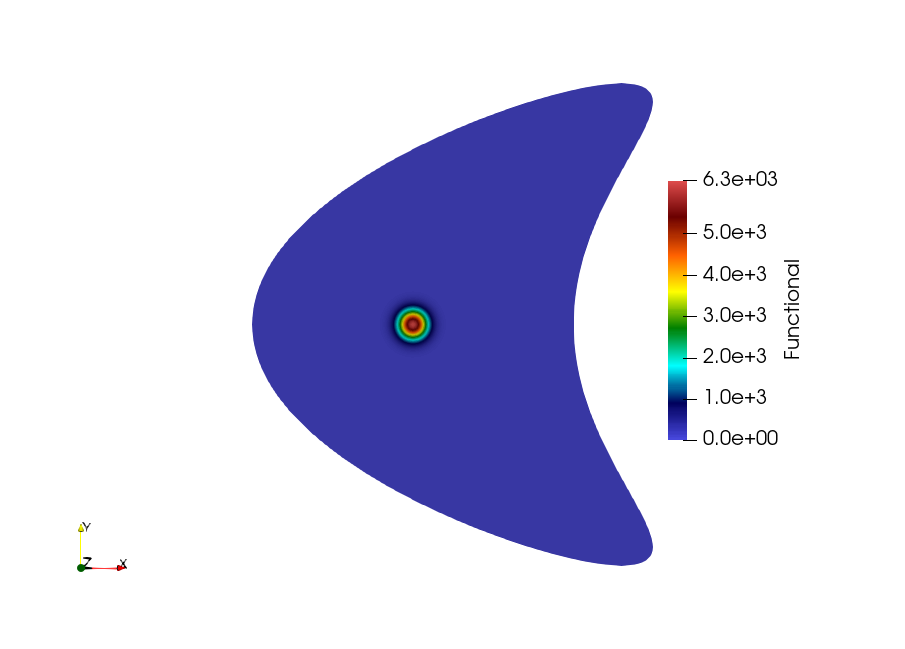}
\caption{\label{fig:data}Visualization of the source term (left) and the functional
\(\phi\) (right).}
\end{figure}

Exemplarily, we consider the source term
\[
f({\bs x})=\frac{1}{2\pi}\big(\log\|{\bs x}+{\bs x}_0\|_2
-\log\|{\bs x}-{\bs x}_0\|_2\big),\quad{\bs x}\in\Rbb^2,
\]
where \({\bs x}_0=[0.1,0]^\intercal\in D_0\). The particular quantity 
of interest \eqref{QoI} involves the localized Gaussian
\[
\phi({\bs x})=\frac{5\pi}{10^4}
e^{-\frac{5}{10^4}\|{\bs x}\|_2^2}.
\]
The reference shape \(D_0\) is given by the kite domain shown in 
Figure~\ref{fig:data} with bounding box \([-0.26,0.39]\times[-0.39,0.39]\).
It is discretized by \(327\,680\) piecewise linear parametric finite 
elements. A visualization of the source term and the functional can 
also been found in the left and right plot of Figure~\ref{fig:data},
respectively.

Finally, in order to define the deformation field, we compute 
the first 130 eigenfunctions $\{\bs \chi_k\}$ of the matrix-valued 
covariance kernel
\[
\Ccal\colon D_0\times D_0\to\Rbb^{2\times 2},\quad
\Ccal({\bs x},{\bs y})\isdef\begin{bmatrix}
10^2e^{-5\|{\bs x}-{\bs y}\|_2^2} & e^{-\|{\bs x}-{\bs y}\|_2^2}\\
e^{-\|{\bs x}-{\bs y}\|_2^2} & 10^2e^{-5\|{\bs x}-{\bs y}\|_2^2},
\end{bmatrix}
\]
discretized on the aforementioned finite element mesh. 
These eigenfunctions are weighted with the singular values
\(\sigma_k\isdef 0.25k^{-3}\) to complement
the random deformation field in accordance with \eqref{eq:pKL}.
\begin{figure}[htb]
\includegraphics[scale=0.18,clip,trim= 200 30 180 30]{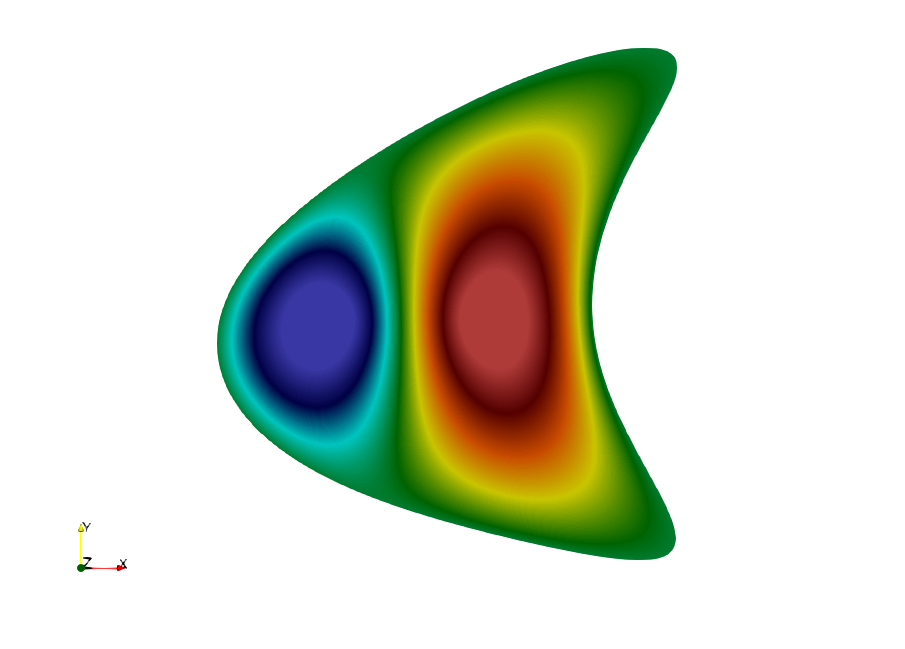}
\includegraphics[scale=0.18,clip,trim= 200 30 180 30]{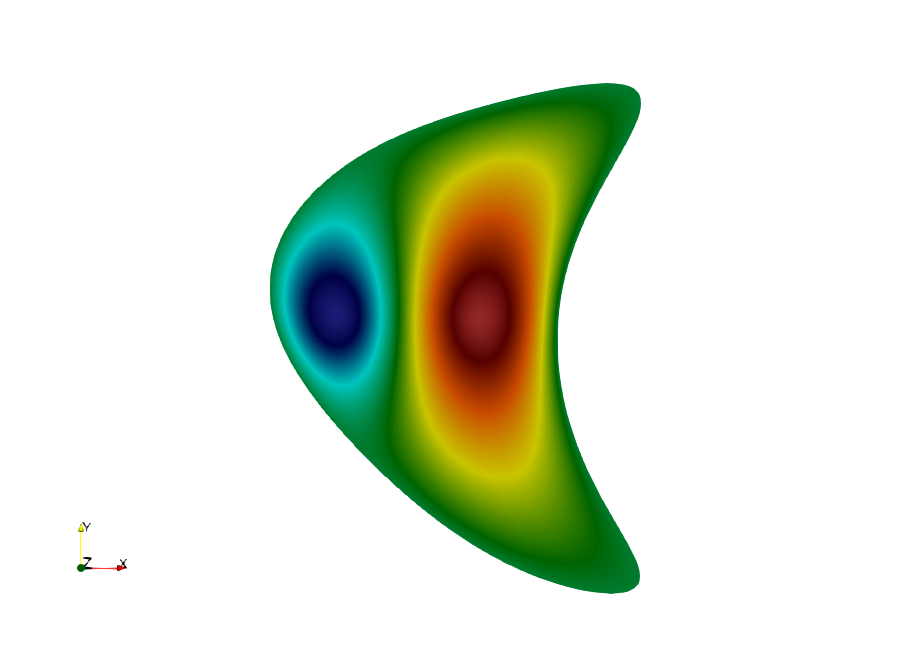}
\includegraphics[scale=0.18,clip,trim= 200 30 180 30]{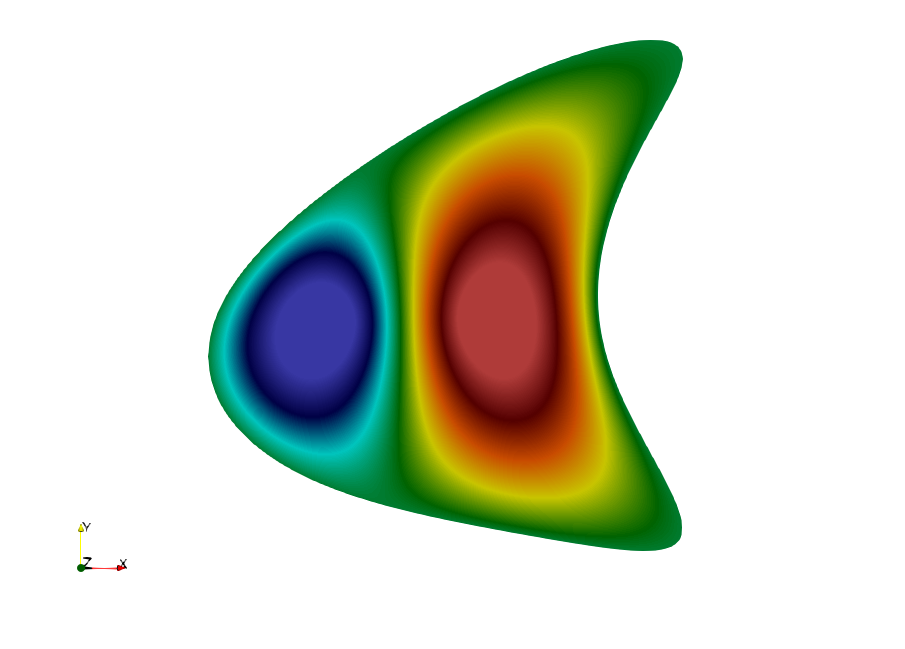}
\includegraphics[scale=0.18,clip,trim= 200 30 180 30]{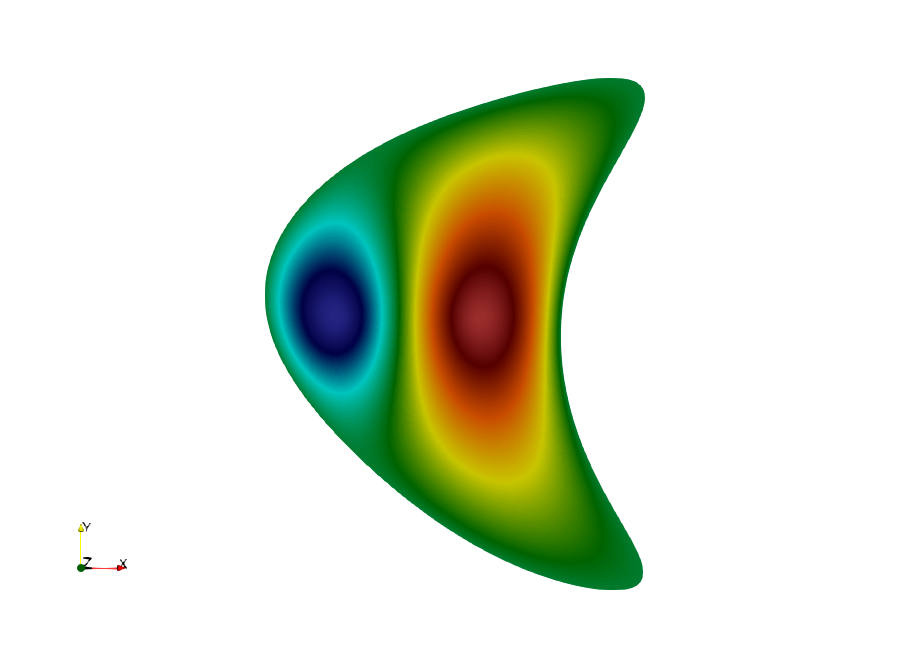}
\caption{\label{fig:reals}Different realizations of the parametric domain
with corresponding solutions.}
\end{figure}
Four particular realizations of the resulting 
random domain are depicted in Figure~\ref{fig:reals}, where color-coding visualizes the corresponding solutions to \eqref{SPDE}. The expectation and the 
standard deviation approximated by the Monte Carlo method can be found in Figure~\ref{fig:meanstd}.

\begin{figure}[htb]
\includegraphics[scale=0.2,clip, trim= 230 60 80 40]{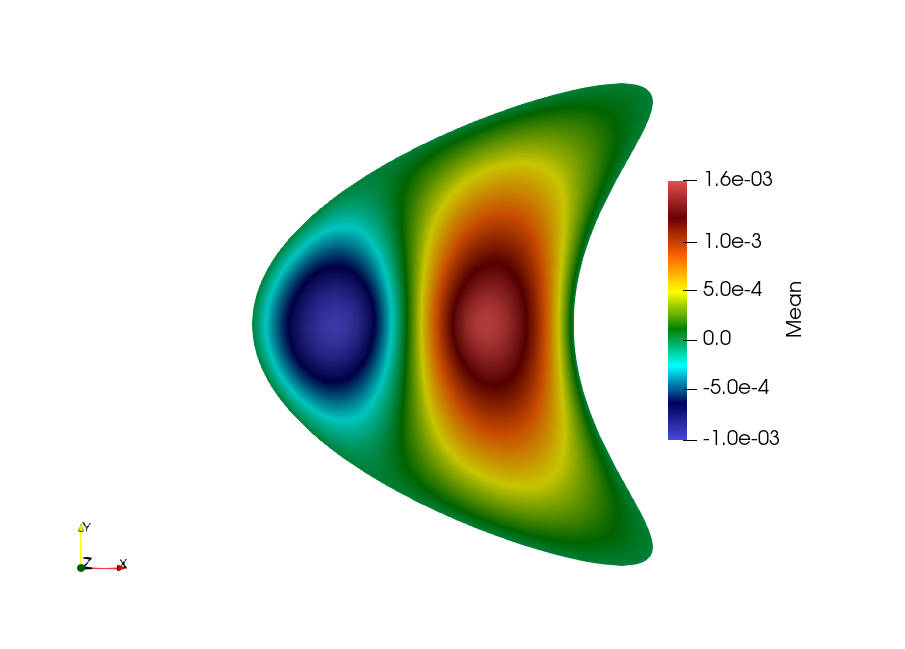}
\qquad\qquad
\includegraphics[scale=0.2,clip, trim= 230 60 80 40]{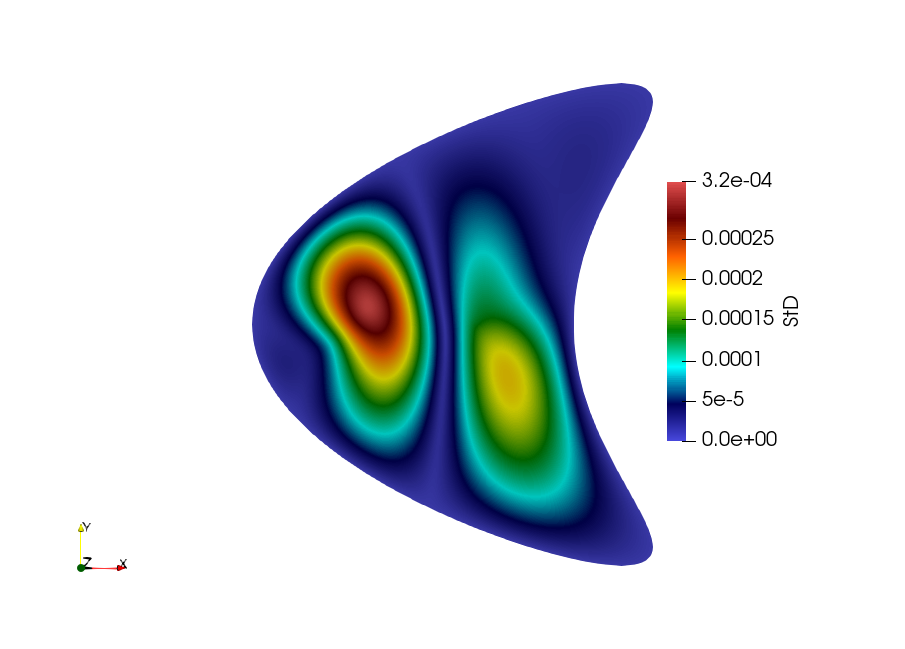}
\caption{\label{fig:meanstd}Expectation (left) and standard deviation (right)
of the
parametric diffusion problem approximated by the Monte Carlo method.}
\end{figure}

We consider $N=849\,375$ data points for the kernel interpolation, 
which are obtained by rescaling
each component of
the Monte Carlo samples \({\bs\zeta}_1,\ldots,{\bs\zeta}_N\) by 
the respective singular value, i.e., \({x}_{i,k}=\sigma_k\zeta_{i,k}\),
resulting in points 
\({\bs x}_1,\ldots{\bs x}_N\in\bigtimes_{k=1}^M[-\sigma_k,\sigma_k]\).
These points are then shifted 
by the choice \(b_k=2\sigma_k\) for all $k=1,\ldots,N$ such that \(X\in\Bcal\), 
see \eqref{eq:anisotropicBox}. 
Moreover, we perform a dimension truncation with
a relative error of \(10^{-3}\), resulting in \(d=20\) dimensions.
As in the benchmarks in the previous paragraph, we apply the 
exponential kernel $\kernel(\bs{x},\bs{y}) = \exp(\|\bs{x}-\bs{y}\|_2/\sigma)$, where
we choose \(10\) different values of the length scale parameter $\sigma$ 
from a logarithmically equispaced 
grid within the interval 
$\big[\max\{10^{-5}, q_{X}\},\operatorname{diam}(B_X)\big]$. 
For the ridge parameters \(\lambda\), we consider \(10\) different 
values, logarithmically equispaced 
within the interval 
$\big[10^{-6}, 10^{-1}\big]$. The 
construction of the cluster tree in the
multipole method is based on the admissibility 
condition \eqref{eq:admissibility} with
\(\eta=0.5\) while the maximum polynomial 
degree for the construction of 
\(\Lambda_{\boldsymbol{\omega},q,d}\) is
\(q=5\). The dimension weights are
computed as in the previous section.

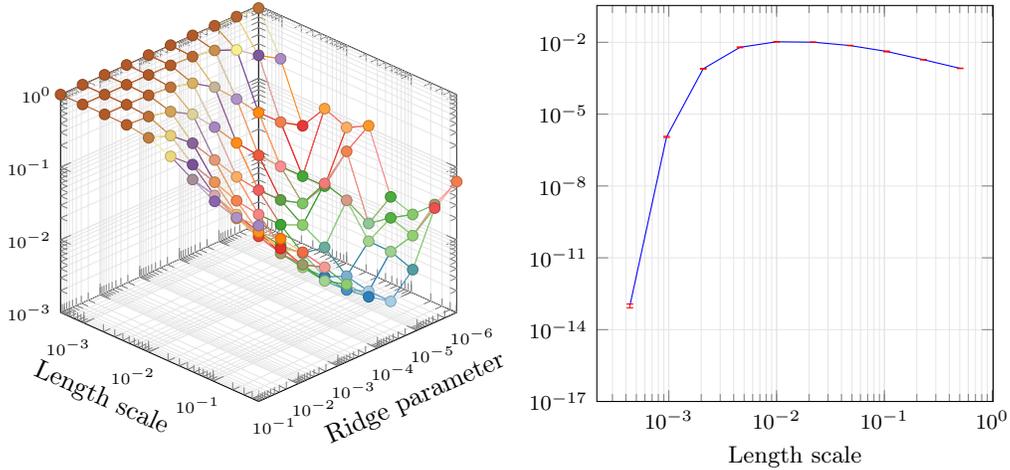
\begin{figure}[htb]
\begin{center}
\pgfplotsset{width=7cm,compat=1.3}
\begin{tikzpicture}
    \pgfplotsset{small}
    \matrix {
\begin{axis}[zmode=log,
xmode = log,
ymode = log,
       zmax = 1,
       zmin = 1e-3,
grid = both,
ticklabel style = {font=\tiny},
y label style={ anchor=center },
xlabel=Length scale,
y label style={rotate=75},
grid style={ultra thin, gray!20},
ylabel=Ridge parameter,
x label style= {rotate=90},
  every axis x label/.append style={sloped, at={(rel axis cs: 0.5, -0.250, 0)}, below},
  every axis y label/.append style={sloped, at={(rel axis cs: 1.250, 0.5, 0)}, below}, 
colormap/Paired,
shader=interp,
y dir = reverse,
view={45}{30},
mesh/cols=10,
mesh/ordering=y varies,
height=6.85cm,width=6.85cm]
 \addplot3+[mesh, scatter] file{UQ_APE.txt};
\end{axis}
& 
\begin{axis}[
xmode = log,
ymode = log,
ymin = 1e-17,
grid,
xtick={1e-4,1e-3,1e-2,1e-1, 1},
y label style={ anchor=center },
xlabel=Length scale,
y label style={rotate=75},
grid=both,
grid style={ultra thin, gray!20},
height=6.85cm,width=6.85cm]
 \addplot[blue,mark size=1.25pt, error bars/.cd, y dir=both,y explicit, error bar style={red}]
coordinates { 
   (4.3541014e-04,9.9706507e-14) +- (0,1.7912980e-14)
   (9.5350291e-04,1.1217673e-06) +- (0,6.0046173e-08)
   (2.0880722e-03,7.8517651e-04) +- (0,2.8099672e-05) 
   (4.5726609e-03,6.1734033e-03) +- (0,2.4829796e-04) 
   (1.0013652e-02,1.0516872e-02) +- (0,2.9263639e-04) 
   (2.1928856e-02,1.0107057e-02) +- (0,1.2002201e-04) 
   (4.8021914e-02,7.3329249e-03) +- (0,1.4544171e-04) 
   (1.0516300e-01,4.1494009e-03) +- (0,1.7357668e-04) 
   (2.3029602e-01,1.8789771e-03) +- (0,3.9762267e-05) 
   (5.0432434e-01,8.1678026e-04) +- (0,1.9089899e-05) 
};
 \end{axis}
\\ };
\end{tikzpicture}
 \caption{\label{fig:UQ_test}
Average prediction error (left) and compression error 
in blue with its standard variation in red (right).}
\end{center}
\end{figure}

The results of the computations are found in 
Figure~\ref{fig:UQ_test}. From the left plot, one figures 
out that the best observed parameter combination is 
$\sigma=5.04\cdot 10^{-1}$ and $\lambda=4.64\cdot 10^{-5}$, 
resulting in an error of \(3.65\cdot 10^{-3}\pm 1.33\cdot 10^{-4}\). 
As in the previous tests, we note from the right-plot of 
Figure~\ref{fig:UQ_test} that the compression error is
noticeably small for small values of the length scale parameter, 
with monotonic growth up to the peak $1.05\cdot 10^{-2}$ for 
\(\lambda=1.00\cdot 10^{-2}\), and then monotonically decreasing.

\section{Conclusion}\label{sec:conclusio}
In the present article, we have proposed a dimension weighted
fast multipole method in the context of scattered data 
approximation. This multipole 
method employs weighted total degree polynomial interpolation 
in the farfield of the kernel function under consideration. 
A rigorous error analysis is provided.
As shown by the extensive numerical
tests, the method is feasible and accurate, while 
enabling to deal with much more dimensions than the 
fast multiple method based on tensor product or
total degree interpolation.

\bibliographystyle{plain}
\bibliography{literature}
\end{document}